\documentclass{amsart}
\usepackage{amsmath,amssymb,amsthm} \usepackage[mathscr]{eucal}
\usepackage[dvipsnames]{xcolor}
\usepackage{amscd} \usepackage{setspace} \usepackage{epsfig}\usepackage{color}
 \usepackage{dsfont}\usepackage{tensor}

\def\Bs{\mathcal{B}}\def\Cs{\mathcal{C}}
\def\Kh{\textit{Kh}}

\def\id{\mathds{1}}\def\Z{\mathds{Z}}\def\F{\mathds{F}}

\newtheorem*{maintheorem}{Main theorem}
\newtheorem*{mainquestion}{Main question}
\newtheorem{theorem}{Theorem}[section]

\newtheorem{obs}[theorem]{Observation}
\newtheorem{prop}[theorem]{Proposition}
\newtheorem{question}[theorem]{Question}

\begin{document}

\title{Plumbing is a natural
operation in Khovanov homology}
\author{Thomas Kindred}
\date{\today}

\begin{abstract}  
Given a connect sum of link diagrams, there is an isomorphism which decomposes {\it unnormalized} Khovanov chain groups for the product in terms of {\it normalized} chain groups for the factors; this isomorphism is straightforward to see on the level of chains.
Similarly, any plumbing $x*y$ of Kauffman states carries an isomorphism of the chain subgroups generated by the enhancements of  $x*y$, $x$, $y$:
  \[
 \Cs_R(x*y)\to
 \left(\Cs_{R,p\to1}(x)\otimes \Cs_{R,p\to1}(y)\right)\oplus\left(\Cs_{R,p\to0}(x)\otimes \Cs_{R,p\to0}(y)\right).
\]
We apply this plumbing of chains 
to prove that every homogeneously adequate state has enhancements $X^\pm$ 
in distinct $j$--gradings whose $\color{gray}A\color{black}$--traces (cf \textsection\ref{S:block}) represent nonzero Khovanov homology classes over $\F_2$, and that this is also true over $\Z$ when all $\color{gray}A\color{black}$--blocks' state surfaces are two--sided.  We construct $X^\pm$ explicitly.  
\end{abstract}

\maketitle


\section{Introduction}\label{S:intro}
Given a link diagram $D\subset S^2$, 
smooth each crossing in one of two ways, 
$\raisebox{-.02in}{\includegraphics[width=.125in]{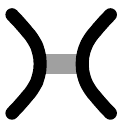}}
\overset{\color{gray}{_{\text{A}}}\color{black}}{\longleftarrow}\raisebox{-.02in}{\includegraphics[width=.125in]{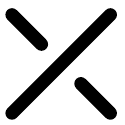}}
\overset{_{\text{B}}}{\longrightarrow}\raisebox{-.02in}{\includegraphics[width=.125in]{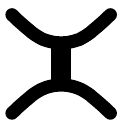}}$. 
The resulting diagram $x$ is called a Kauffman \textbf{state} of $D$ and consists of 
{\it state circles} joined by $\color{gray}A\color{black}$-- and $B$--labeled arcs, one from each crossing.
 \textbf{Enhance} $x$ by assigning each state circle a binary label: 
   $\color{ForestGreen}\bigcirc \color{black} 
\overset{_{\color{ForestGreen}1\color{black}}}{\longleftarrow}
\bigcirc 
\overset{_{\color{NavyBlue}0\color{black}}}{\longrightarrow}
 \color{NavyBlue}\bigcirc\color{black}$, and let $R$ be a ring with 1. The enhanced states from $D$ form an $R$--basis for a bi-graded 
{chain complex} $\Cs_R(D)=\bigoplus_{i,j\in\Z}\Cs_R^{i,j}(D)$, 
which has a differential  $d$ of degree $(1,0)$; 
the resulting homology groups 
are link--invariant. 
Khovanov homology {\it categorifies} the 
Jones polynomial in the sense that the latter
is the graded euler characteristic of the former \cite{jones,khov,viro}. Section \ref{S:viro} reviews Khovanov homology in more detail. 

What do (representatives of) nonzero Khovanov homology classes look like?  The simplest examples come from
 adequate all--$\color{gray}A\color{black}$ states $x_{\color{gray}A\color{black}}$ and adequate all--$B$ states $y_B$: the all--$\color{ForestGreen}1$  \color{black} enhancement of $x_{\color{gray}A\color{black}}$ and the all--\color{NavyBlue}$0$  \color{black} enhancement of $y_B$ are nonzero cycles with any coefficients. Further, any enhancement of $y_B$ with exactly one $\color{ForestGreen}1$\color{black}--label is a nonzero cycle over any $R$ in which 2 is not a unit; and the sum of all enhancements of $x_{\color{gray}A\color{black}}$ with exactly one \color{NavyBlue}0\color{black}--label is a nonzero cycle over $R=\F_2$. 
 
Intriguingly, such states $x_A$, $y_B$ are \textit{essential} in the sense that 
their state surfaces are 
incompressible and $\partial$--incompressible \cite{ozawa}. Does Khovanov homology detect essential surfaces in any more general sense?  Letting  $\Cs_R(x)$ denote the submodule of $\Cs_R(D)$ generated by the enhancements of any state $x$ of $D$, we ask: 

\begin{mainquestion}
For which essential states $x$ 
does $\Cs_R(x)$ contain a 
nonzero homology class?
\end{mainquestion}

As this inquiry depends explicitly on the diagram, the chief motivation is not Khovanov homology in the abstract, but rather a geometric question: in what sense does Khovanov homology detect essential surfaces? 

Which states $x$ are essential? 
A necessary condition is that $x$ must be adequate.  
For a sufficient condition, let ${G_x}$ denote the graph obtained from $x$ by collapsing each state circle to a vertex (each 
crossing arc is then an edge).  Cut ${G_x}$ all at once along its cut vertices (ones whose deletion disconnects $G_x$), and consider the resulting connected components; the corresponding subsets of $x$ are called \textbf{blocks}. 
The state $x$ decomposes under \textbf{plumbing} (of states) into these blocks, and the state surface from $x$ decomposes under plumbing (of surfaces) into the blocks' state surfaces, each of which is a checkerboard surface for its block's underlying link diagram. Section \ref{S:background} reviews state surfaces and plumbing in more detail.

If each block of $x$ is essential, then $x$ is essential too, as plumbing respects essentiality \cite{gab1,ozawa}.   In particular, if each block of $x$ is adequate and either all--$A$ or all--$B$, then $x$ is essential, and is called \textbf{homogeneously adequate} \cite{crom,purcell}.   Our main result states that Khovanov homology over $\F_2=\Z/2\Z$ detects all such states:

\begin{maintheorem}\label{T:main}
If $x$ is a homogeneously adequate state, then $\Cs^{i_x,j_x\pm 1}_{\F_2}(x)$ both contain  (representatives of) nonzero homology classes. If also $G_{x_{\color{gray}A\color{black}}}$ is bipartite, then $\Cs^{i_x,j_x\pm 1}_{\Z}(x)$ 
contain such classes as well.
\end{maintheorem} 
Here, $i_x,j_x$ are integers that depend only on $x$, and $x_{\color{gray}A\color{black}}$ denotes the union of the $\color{gray}A\color{black}$--blocks of $x$. (We will define $x_B$ analogously; this is consistent with the earlier notation $x_{\color{gray}A\color{black}}$, $y_B$.)  The bipartite condition on $G_{x_{\color{gray}A\color{black}}}$ is equivalent to the condition that the state surfaces from the $\color{gray}A\color{black}$--blocks of $x$ are all two--sided. 
 In general, the condition of homogeneous adequacy is sensitive to changes in the link diagram, as are the homology classes from the main theorem, in the sense that Reidemeister moves generally do not preserve the fact that these classes have representatives in some $\Cs_R(x)$.
In the adequate all--$\color{gray}A\color{black}$ case $x=x_{\color{gray}A\color{black}}$, with ${G_x}$ bipartite, the link $L$ can be oriented so that the diagram $D$ is {\it positive}; if this $D$ is a closed braid diagram, then the class from $\Cs_{\Z}^{i_x,j_{x-1}}(x)$ is Plamenevskaya's {\it distinguished element} $\psi(L)$ \cite{plamen}.

Section \ref{S:plumb} develops the operation $*$ of plumbing on Khovanov chains in order to prove the main theorem by induction, extending the all--$\color{gray}A\color{black}$ and all--$B$ cases to the homogeneously adequate case in general. The idea is simple: glue two enhanced states along a state circle where their labels match so as to produce a new enhanced state; then extend linearly. 
Unfortunately, even simplest case of plumbing---connect sum, $\natural$---reveals a technical wrinkle: the differential sometimes changes the labels on the state circle along which the two plumbing factors are glued together, upsetting the compatibility required for the plumbing. The workaround is to specify, by a rule of trumps, whether the labels on the first plumbing factor override those on the second or vice-versa.  The upshot is a useful 
 identity:
\[{d}(X*Y)={d} X \tensor[_\diamondsuit]{*}{}Y+\left(-1\right)^{| \raisebox{-.02in}{\includegraphics[width=.1in]{ASmooth.pdf}}|_x}
 X\tensor[]{*}{_\diamondsuit}{d} Y.\]
 Roughly, this states that plumbing $*$ behaves like an exterior product followed by interior multiplication.  The effect of this workaround is that the inductive proof of the main theorem, although 
hopefully instructive, is somewhat complicated.  Section \ref{S:direct} offers an easier, direct proof.
Section \ref{S:final} gives two easy examples of inessential states $x$ with nonzero $\Cs_R(x)$,
constructs a class of {\it non-homogeneous} essential states $y$, asks whether $\Cs_R(y)$ is nonzero for $y$ in this class, and ends with further open questions.

\underline{Notation:} For a diagram $Z$ of any sort and any feature $\odot$ which may appear in such diagram, $\left|\odot\right|_Z$ denotes the number of $\odot$'s in $Z$. For example, if $D$ is a link diagram, then $\left| \raisebox{-.02in}{\includegraphics[height=.125in]{Crossing.pdf}}\right|_D$ counts the crossings in $D$. 

\section{Khovanov homology of a link diagram, after Viro}\label{S:viro}
\subsection{Enhanced states.}\label{S:enhanced}
Index the crossings of a (connected) link diagram $D$ as $c^1,\hdots, c^{| \raisebox{-.02in}{\includegraphics[width=.1in]{Crossing.pdf}}|_D}$, and
  make a binary choice at each crossing: 
  $\raisebox{-.02in}{\includegraphics[width=.125in]{ASmooth.pdf}}
\overset{\color{gray}{_{\text{A}}}\color{black}}{\longleftarrow}\raisebox{-.02in}{\includegraphics[width=.125in]{Crossing.pdf}}
\overset{_{\text{B}}}{\longrightarrow}\raisebox{-.02in}{\includegraphics[width=.125in]{BSmooth.pdf}}$. The resulting diagram $x\subset S^2$ is called a Kauffman \textbf{state} of ${D}$ and consists of $\left|\bigcirc\right|_x$ {\it state circles} 
 joined by \color{gray}$A-$
 \color{black} and  \color{black} $B-$ \color{black} labeled arcs, one from each crossing. 
 Index the state circles of $x$ as $x_1,\hdots, x_{\left| \bigcirc\right|_x}$, and
 \textbf{enhance} $x$ by making a binary choice at each state circle, $x_r$: 
$\color{ForestGreen}\bigcirc \color{black} 
\overset{_{\color{ForestGreen}a_r=1\color{black}}}{\longleftarrow}
\bigcirc 
\overset{_{\color{NavyBlue}a_r=0\color{black}}}{\longrightarrow}
 \color{NavyBlue}\bigcirc\color{black}$. 
  Letting $R$ be a ring with 1, 
  define $\Cs_R(x)$ to be the $R$--module generated by the enhancements of $x$. Let $V:=R[q]/(q^2)$, and associate $\Cs_R(x)$ with $V^{\otimes\left| \bigcirc\right|_x}$ 
    by identifying each enhancement of $x$ 
 with the simple tensor 
$q^{a_1}\otimes\cdots\otimes q^{a_{\left|\bigcirc\right|_x}}$.
Define: 
\[ \Cs_R(D):=\bigoplus_{\small{\text{states }
 {x}\text{ of }D}}\Cs_R(x)
 =
  \bigoplus_{\small{\text{states }
 {x}\text{ of }D}}V^{\otimes\left| \bigcirc\right|_x}
.\]


\subsection{Grading.}\label{S:grading}
The 
{\it writhe} of an oriented diagram $D$ is ${w_D}=\left| \raisebox{-.02in}{\includegraphics[width=.125in]{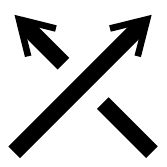}}\right|_D-\left| \raisebox{-.02in}{\includegraphics[width=.125in]{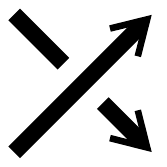}}\right|_D$. 
  For each state $x$ of ${D}$, let 
${\sigma_x}:=\left| \raisebox{-.02in}{\includegraphics[width=.125in]{ASmooth.pdf}}\right|_x-\left| \raisebox{-.02in}{\includegraphics[width=.125in]{BSmooth.pdf}}\right|_x$ and 
$ {{i}_x}:=\frac{1}{2}({w_D}-{\sigma_x}).$
For any enhancement $X$ of $x$, define
${\tau_X}=|\color{ForestGreen}\bigcirc\color{black}|_X-|\color{NavyBlue}\bigcirc\color{black}|_X=-{\left|\bigcirc\right|_x}+2\sum_{r}a_r$ and ${{j}_X}:={w_D}+{{i}_x}-{\tau_X}.$ 
The $R$--module $\Cs_R(D)$ carries a bi-grading $\Cs_R(D)=\bigoplus_{i,j}\Cs^{i,j}_R(D)$, where each $\Cs^{i,j}_R(D)$ is generated by the enhancements $Y$ of states $y$ of $D$ with $i={{i}_y}=:i_Y$ and $j={{j}_Y}$. 

\begin{figure}
\begin{center}
\includegraphics[height=4in]{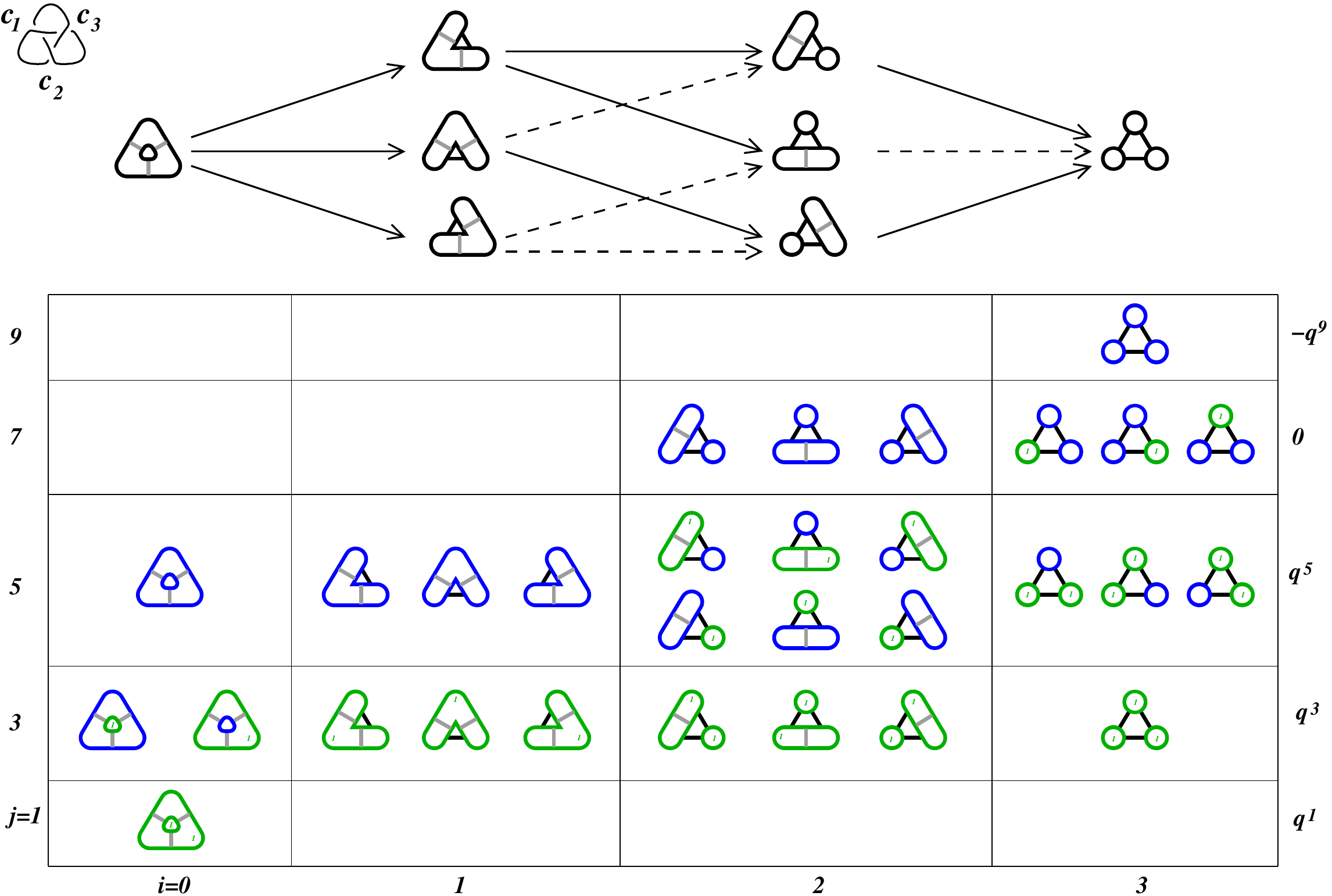}
\caption{The Jones polynomial $V_{K}(q)=q+q^3+q^5-q^9$ of the \textsc{RH} trefoil  via Khovanov chains.}\label{Fi:TrefGridJones}
\end{center}
\end{figure}

The Jones polynomial $V_K(q)$ of an oriented link $K$, unnormalized such that $V_{\text{unknot}}(q)=q+q^{-1}$
, is 
given 
by Kauffman's  state sum formula \cite{jones,kauff}.
Enhancement \textsc{foils} this  formula in order to express the Jones polynomial as the graded euler characteristic of $\Cs_R(D)$ (cf Figure \ref{Fi:TrefGridJones}) \cite{viro}:
\[V_K(q)=q^{3{w_D}}\sum_{\text{states }x}
\left(-1\right)^{{{i}_x}}\left(q+q^{-1}\right)^{\left| \bigcirc\right|_x}=\sum_{\text{enhanced states }X}
\left(-1\right)^{{{i}_X}}q^{{{j}_X}}=\sum_{i,j\in\Z}\left(-1\right)^iq^j\,\text{rk}(\Cs^{i,j}(D)).\]
\subsection{Homology.}\label{S:maps}
With $X$ an enhanced state from a link diagram $D$, define the differential $d_{c^t}X$ of $X$ at each crossing $c^t$ of $D$ by the incidence rules in Figure \ref{Fi:incidence}. (If $x$ has a $B$--smoothing at $c^t$, then ${d}_{c^t}X=0$.) 
In general, the differential of an enhanced state $X\in \Cs^{i,j}_R(D)$ equals the sum ${d} X=\sum_{t}\left(-1\right)^{| \raisebox{-.02in}{\includegraphics[width=.1in]{ASmooth.pdf}}|_{X}^t}{d}_{c^t}Y\in \Cs^{i+1,j}_R(D)$, where $|\raisebox{-.02in}{\includegraphics[height=.125in]{ASmooth.pdf}}|_{X}^t$ is the number of crossings $c^s$ with $s<t$ at which $X$ has an $A-$smoothing. 
When $R=\F_2$, the differential is simply
${d} X=\sum_{t}{d}_{c^t}X.$

Extend $R$--linearly  to obtain the 
\textbf{differential}  ${d}: \Cs_R(D)\to \Cs_R(D)$, which has degree $(1,0)$ and obeys ${d}\circ{d} \equiv 0$, 
giving 
$\Cs_R(D)$  the structure of a \textbf{chain complex}.  A chain $X\in \Cs_R(D)$---ie an $R$--linear combination of enhanced states from $D$---is called \textit{closed} if $dX=0$ and  \textit{exact} if $X=dY$ for some $Y\in \Cs_R(D)$; closed chains are called \textit{cycles}, exact chains \textit{boundaries}. Take {\it cycles mod boundaries} to define Khovanov's \textbf{homology groups} $\Kh_R(D)=\text{ker}({d})/\text{image}({d})$, which are link--invariant.
 
\begin{figure}
\begin{center}
\includegraphics[width=6.5in]{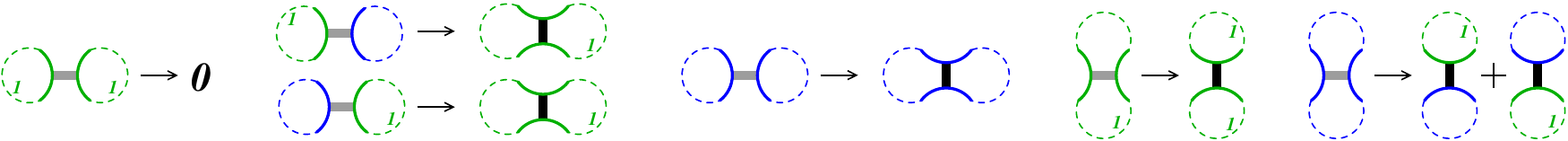}
\caption{Incidence rules for ${d}$: $q\otimes q\mapsto 0$; $1\otimes q,q\otimes1\mapsto q$; $1\otimes 1\mapsto 1$; $q\mapsto q\otimes q$; $1\mapsto q\otimes 1+1\otimes q$.}
\label{Fi:incidence}
\end{center}
\end{figure}

The \textbf{augmentation} map $\varepsilon:\Cs_R(D)\to R$ is the $R$--linear map that sends each enhanced state $X$ to $1$. 
A subset $\Bs\subset \Cs_R(D)$ is called {\it primitive} if, whenever $r\in R$, $X\in \Cs_R(D)$, and $rX\in\Bs$, also $uX\in \Bs$ for some unit $u\in R$.  For example, a collection of enhanced states is primitive. If $\Bs\subset \Cs_R(D)$ is primitive, then the \textbf{projection} map $\pi_{\Bs}:\Cs
_R(D)\to\Cs
_R(D)$ is the $R$--linear map that sends each chain $X$ to itself when $X$ is in the $R$--span of $\mathcal{B}$
and to $0$ otherwise. 
\subsection{Normalization}
\label{S:normal}
%
%
Let $D$ be a link diagram, $R$ a ring with 1, and $p$ a point on $D$ away from crossings. For each state $x$ of $D$, define $\Cs_{R,\color{ForestGreen}p\to1\color{black}}(x)$, $\Cs_{R,\color{NavyBlue}p\to0\color{black}}(x)$ to be the subcomplexes of $\Cs_R(x)$ generated by those enhancements of $x$ in which the state circle containing the point $p$ has the indicated label. 
Note that $\Cs_R(x)=\Cs_{R,\color{ForestGreen}p\to1\color{black}}(x)\oplus \Cs_{R,\color{NavyBlue}p\to0\color{black}}(x)$ and thus
\[\Cs_R(D)=\bigoplus_{\text{states }x\text{ of }D}\Cs_{R,\color{ForestGreen}p\to1\color{black}}(x)\oplus \Cs_{R,\color{NavyBlue}p\to0\color{black}}(x).\]
Define the subcomplexes $\Cs_{R,\color{ForestGreen}p\to1\color{black}}(D):=\bigoplus_{x}\Cs_{R,\color{ForestGreen}p\to1\color{black}}(x)$ and $\Cs_{R,\color{NavyBlue}p\to0\color{black}}(D):=\bigoplus_{x}\Cs_{R,\color{NavyBlue}p\to0\color{black}}(x)$, with a shift of $\pm 1$ in the 
$j$--grading due to omitting the state circle containing the point $p$ from the definitions of $\tau$ and thus of $j$, and with differentials obtained by restricting ${d}$ as follows.  If $X\in \Cs_{R,\color{ForestGreen}p\to1\color{black}}(D)$, $Y\in \Cs_{R,\color{NavyBlue}p\to0\color{black}}(D)$ are enhanced states, then their respective differentials in $\Cs_{R,\color{ForestGreen}p\to1\color{black}}(D)$, $\Cs_{R,\color{NavyBlue}p\to0\color{black}}(D)$ are
\[
\sum_{c^t}\left(-1\right)^{| \raisebox{-.02in}{\includegraphics[width=.1in]{ASmooth.pdf}}|_{X}^t}
\cdot\pi_{\Cs_{R,\color{ForestGreen}p\to1\color{black}}(D)}\circ{d}_{c^t}X,
\hspace{.5in}
\sum_{c^t}\left(-1\right)^{| \raisebox{-.02in}{\includegraphics[width=.1in]{ASmooth.pdf}}|_{Y}^t}
\cdot\pi_{\Cs_{R,\color{NavyBlue}p\to0\color{black}}(D)}\circ{d}_{c^t}Y.\]
In other words, the differentials of $X$ in $\Cs_{R,\color{ForestGreen}p\to1\color{black}}(D)$, $\Cs_{R,\color{NavyBlue}p\to0\color{black}}(D)$ are the same sums of enhanced states as $dX$ in $\Cs_{R}(D)$, subject to the extra condition on the label at $p$.
The graded euler characteristics of the resulting homology groups $\Kh_{R,\color{ForestGreen}p\to1\color{black}}(D)$, $\Kh_{R,\color{NavyBlue}p\to0\color{black}}(D)$ both equal the normalized Jones polynomial, $V_K(D)/\left(q^{-1}+q\right).$
%
\section{Further background}\label{S:background}
%
\subsection{Blocks and zones.}\label{S:block}
 Associate to each state $x$ a \textit{state graph} ${G_x}$ by collapsing each state circle of $x$ to a point; maintain the $\color{gray}A\color{black}$-- and $B$--labels on the edges of ${G_x}$, which come from the crossing arcs in $x$. Cut ${G_x}$ simultaneously along all its cut vertices
 .  The subsets of $x$ corresponding to the resulting (connected) components are called the \textbf{blocks} of $x$. 
  If no block 
contains both $\color{gray}A\color{black}$-- and $B$--type crossing arcs, then $x$ is called 
\textit{homogeneous}  \cite{ozawa,crom}. 
A state $x$ is called \textit{adequate} if each crossing arc joins distinct state circles. If $x$ is both adequate and homogeneous, it is called \textbf{homogeneously adequate}.

%
%
\begin{figure}
\begin{center}
\includegraphics[width=6.5in]{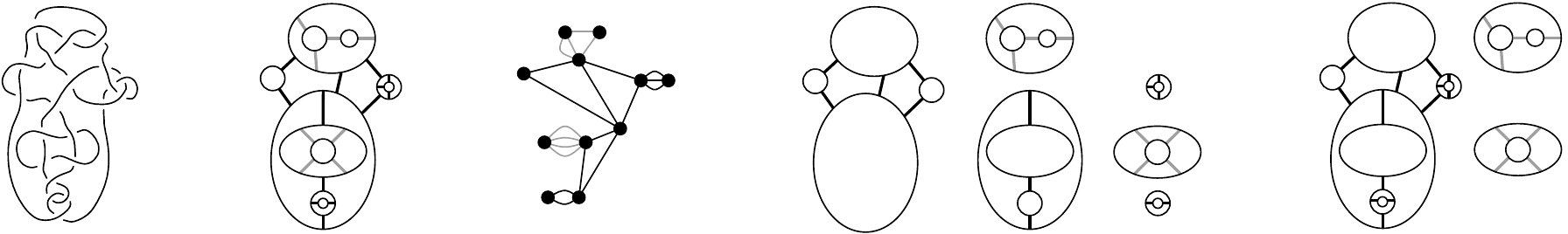}
\caption{From left to right: a link diagram $D$ with a homogeneously adequate state $x$, and its graph ${G_x}$, blocks, and zones.}\label{Fi:blocks}
\end{center}
\end{figure}
%
%

 Given any state $x$, define $x_{\color{gray}A\color{black}}$ to be the union of all $\color{gray}A\color{black}$--type crossing arcs and their incident state circles; define $x_B$ analogously.  If $x$ is homogeneous, $x_{\color{gray}A\color{black}}$ and $x_B$ are the respective unions of the $\color{gray}A\color{black}$-- and $B$--type blocks of $x$. In this case,
define the $\color{gray}A\color{black}$-- and $B$--type homogeneous \textbf{zones} of $x$ to be the components of $x_{\color{gray}A\color{black}}$ and $x_B$, respectively; {call these {\it $\color{gray}A\color{black}$--zones} and {\it $B$--zones} for short}  (cf Figure \ref{Fi:blocks}). 

Define the equivalence relations $\sim_{\color{gray}A\color{black}}$, $\sim_B$ on enhanced states to be
 generated by $
 \raisebox{-.02in}{\includegraphics[width=.125in]{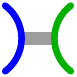}}
 \sim_{\color{gray}A\color{black}}
 \raisebox{-.02in}{\includegraphics[width=.125in]{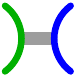}}$ 
 and $
 \raisebox{-.02in}{\includegraphics[width=.125in]{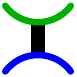}}
 \sim_B
 \raisebox{-.02in}{\includegraphics[width=.125in]{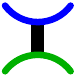}}$, respectively. 
 Let $[X]_{\color{gray}A\color{black}},\left[X\right]_B$ denote the associated equivalence classes. Note that $[X]_{\color{gray}A\color{black}},\left[X\right]_B\subset\Cs_{R}^{i_X,j_X}(x)$. 
%
%
\begin{prop}\label{P:1}
If $X$ enhances a homogeneous state 
$x$, then $[X]_{\color{gray}A\color{black}}\cap [X]_B=\{X\}$. 
\end{prop}
%
%
\begin{proof} Suppose $Y\in [X]_{\color{gray}A\color{black}}\cap [X]_B$ is an enhanced state. Deduce from $X\sim_{\color{gray}A\color{black}}Y$ that $X$, $Y$ are identical in $x\setminus x_{\color{gray}A\color{black}}$, ie  that each state circle of $x_B\setminus x_{\color{gray}A\color{black}}$ has the same 
label in $X$, $Y$. Likewise, $X\sim_BY$ implies that $X$, $Y$ are identical in $x\setminus x_B$. Hence, $X$ and $Y$ are identical in all of $(x\setminus x_{\color{gray}A\color{black}})\cup (x\setminus x_B)=x\setminus (x_{\color{gray}A\color{black}}\cap x_B)$.  
The fact that each 
zone of $x$ has as many \color{NavyBlue}0\color{black}-- (and \color{ForestGreen}1\color{black}--) labeled state circles in $X$ and $Y$ implies that \textit{innermost} circles of $x_{\color{gray}A\color{black}}\cap x_B$ are identical in $X$, $Y$; induction on height in $x_{\color{gray}A\color{black}}\cap x_B$ completes the proof.
\end{proof}
%
%
Define the \textbf{$\color{gray}A\color{black}$--trace} over $R=\F_2$ of any enhanced state $X$ to be $\text{tr}_{\F_2}{X}:=\sum_{Y\sim_{\color{gray}A\color{black}}X}Y$. (The term is chosen in rough analogy with the {\it field trace}; we find no use for an analogous notion of $B$--trace.) To extend this notion to $R=\Z$, suppose $X\sim_{\color{gray}A\color{black}}Y$ are enhanced states; if every non-bipartite $\color{gray}A\color{black}$--zone is all--\color{ForestGreen}1\color{black}, define $\text{sgn}(X\to Y)$ to be $1$ or $-1$ according to whether an even or odd number of $\raisebox{-.02in}{\includegraphics[width=.125in]{KhovAEquiv1.pdf}}
\leftrightarrow \raisebox{-.02in}{\includegraphics[width=.125in]{KhovAEquiv2.pdf}}$ moves
 take $X$ to $Y$. 
Define the {$\color{gray}A\color{black}$--trace} over $\Z$ of such an enhanced state $X$ 
to be  
$\text{tr}_\Z{X}:=\sum_{Y\sim_{\color{gray}A\color{black}}X}\text{sgn}(X\to Y)Y$.
The notion of {$\color{gray}A\color{black}$--trace} is generic to the main question in the following sense:
%
%
\begin{obs}\label{O:generic}
Over $R=\F_2$ (resp. $R=\Z$), every cycle $X\in
\Cs_{R}(x)$ is a sum of traces, $X=\sum_{r}\text{tr}_{R}X_r$,   
and each component of $x_{\color{gray}A\color{black}}$ is adequate and either all--\color{ForestGreen}1 \color{black} or (bipartite) with one \color{NavyBlue}0\color{black}--labeled circle. 
\end{obs}
%
%

\subsection{State surfaces}\label{S:statesurface}

Given a link diagram $D$ on $S^2\subset S^3$, embed the underlying link $L$ in $S^3$ by inserting tiny, disjoint balls $\bigsqcup C^t=C$ at the crossing points $c^t$ and pushing the two  arcs of $D\cap C^t$ to the hemispheres of $\partial C^t\setminus S^2$ indicated by the over--under information at $c^t$.  In this setup, the states of $D$ are precisely the closed 1--manifolds $L\cap S^2\subset x\subset (L\cup \partial C)\cap S^2$. 
Given a state $x$ in this setup,  $x\cup L$ intersects each $\partial C^t$ in a circle.  Cap each such circle with a disk in $C^t$, called a {\it crossing band}, and cap the state circles of $x$ with disks whose interiors are disjoint from one another, all on the same side of $S^2\cup C$. The resulting unoriented surface $F_x$ \textit{spans} $L$, meaning that $\partial F_x=L$, and is called the \textbf{state surface} from $x$.

When $L$ is a knot, its linking number with a co-oriented pushoff $\hat{L}_i$ in $F_x$ is called the \textit{boundary slope} of $F_x$ and equals $2i_x$. 
%
\begin{figure}
\begin{center}
\includegraphics[width=3.5in]{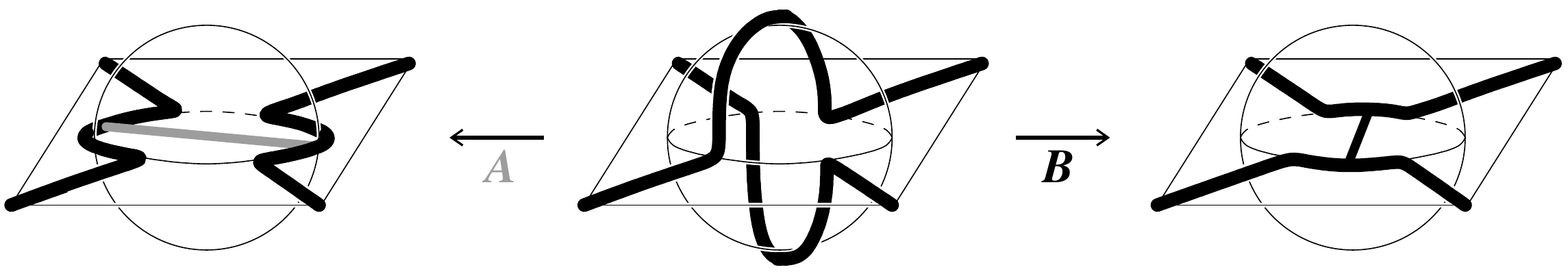}
\caption{Use crossing balls $C=\bigsqcup C^t$ to embed a link and its states in $(S^2\setminus C)\cup\partial C$.}\label{Fi:BubbleSmooth}
\end{center}
\end{figure}
When $L$ is an oriented link, $2i_x$ is the sum of the component--wise boundary slopes of $F_x$, do not depend on the orientation on $L$, and twice the link components' pairwise linking numbers, which do:
\[2i_x=\sum_{\text{link components }L_i}\text{lk}(L_i,\hat{L}_i)+\sum_{\text{ordered pairs of distinct link components }(L_i,L_j)}\text{lk}(L_i,L_j).\]
Correspondences between (enhanced) states and surfaces invite geometric interpretations of Khovanov homology. The author plans to discuss these in a future paper. 
%
\subsection{Plumbing}
In the context of 3-manifolds, {\it plumbing} or {\it Murasugi sum}, is an operation on states, links, and spanning surfaces.  
{ Plumbing two states} $x$, $y$ simply involves gluing these states along a single state circle in such a way that the resulting diagram 
is also a state. Such a plumbing $x*y=z$ is \textit{external} in the sense that it depends on a gluing map ${g}:(S^2,x)\sqcup (S^2, y)\to (S^2,z)$ (cf Figure \ref{Fi:PlumbMap}); the notation $x\overset{{g}}{*}y$ makes this dependence explicit.  

The plumbed state $z=x\overset{{g}}{*}y$ \textit{de--plumbs} as a gluing of the states $g(x)$, $g(y)$ along the state circle $g(x)\cap g(y)$.  Viewing the { plumbing factors} $g(x)$, $g(y)$ as subsets of $z$ and identifying $g(x)$ with $x$, $g(y)$ with $y$ in the obvious way, denote this de-plumbing by $z=x*y$.  This (de--)plumbing $z=x*y$ is \textit{internal} in the sense that $x$ and $y$ are subsets of $z$, and so no extra gluing information is needed. The distinction between internal and external plumbing, taken in analogy with internal and external free products of groups, will help with labeling; usually the distinction is immaterial and we make no comment.

 If $x
 {*}y=z$ is a plumbing of states, then there is an associated {plumbing of link diagrams}, $D_x*D_y=D_z$, and an associated {plumbing of the underlying links}. 
There is also an associated {plumbing of state surfaces}, $F_x*F_y=F_z$; here is how this works.
 Viewing $x*y=z$ as an internal plumbing, let $z_0$ be the state circle comprising $x\cap y$, and let $U$ be the disk that $z_0$ bounds in  the state surface $F_z$. 
There is an embedded sphere $Q\subset S^3$ transverse to the projection sphere $S^2$ with $Q\cap F_z=U$ and $Q\cap S^2=z_0$; let $B_x$, $B_y$ denote the (closed) balls into which $Q$ cuts $S^3$, such that $x\subset B_x$, $y\subset B_y$. The surfaces $F_x:=F_z\cap B_x$, $F_y:=F_z\cap B_y$ are the state surfaces for $x$, $y$, respectively, and plumbing these surfaces along $Q$ produces $F_x*F_y=F_z$.

\begin{figure}
\begin{center}
\includegraphics[width=4.5in]{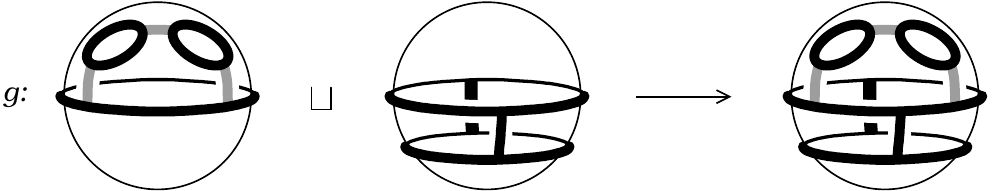}\\
\caption{A gluing map ${g}:(S^2,x)\sqcup(S^2,y)\to (S^2,z)$ for a plumbing  of states $x*y=z$.}\label{Fi:PlumbMap}
\end{center}
\end{figure}

For general interest, we briefly describe two more general notions of (de-){plumbing of spanning surfaces}. First, suppose $F$ spans a link $K\subset S^3$ 
and  $Q\subset S^3$ is a sphere which 
intersects $F$ (non--tranversally) in a disk $U=Q\cap F$. 
If $B_0$, $B_1$ are the (closed) balls into which $Q$ cuts $S^3$ and $F_0=B_0\cap F$, $F_1=B_1\cap F$, so that $F_0\cap F_1=F\cap B_0\cap B_1=F\cap Q=U$, then the sphere $Q$ is said to {\it de-plumb} $F$ as $F=F_0*F_1$.  

A second, more general notion of plumbing, better suited for iteration, views a regular neighborhood of $\text{int}(F)$ in the link complement $S^3\setminus K$ in terms of 
a line bundle $\rho:N\to \text{int}(F)$ and allows de-plumbing along a sphere $Q$ which 
(i) is transverse in $S^3$ to $K$ and $F$, and in $S^3\setminus K$ to the fibers of  $\rho$; and which (ii) 
intersects $F$ in a disk $U$ which is (the image of) a local section of $\rho$. Letting $B_0$, $B_1$ denote the balls into which $Q$ cuts $S^3$, the resulting plumbing factors are $\rho(F\cap B_0)$, $\rho(F\cap B_1)$.  
%

\section{Plumbing Khovanov chains}\label{S:plumb}
The all--$\color{gray}A\color{black}$ state of a link diagram $D$ is always homogeneous; if this state is adequate, then $D$ is called $\color{gray}A\color{black}$--adequate; $B$--adequacy is defined analogously.  It is easy to see, recalling Figure \ref{Fi:TrefGridJones}, that Khovanov homology over any coefficient ring $R$ with $1$ detects any $\color{gray}A\color{black}$-- or $B$--adequate state $x$, in the sense that $\Cs_R(x)$ contains a non--exact cycle. 
The main theorem extends this fact in case $R=\F_2$ to all homogeneously adequate states, and in case $R=\Z$ (or any other ring with $1$ in which $2$ is not a unit) to such states whose $\color{gray}A\color{black}$--zones are bipartite. The inductive key for extending in this way is the operation of plumbing on Khovanov chains.
\begin{figure}
\begin{center}
\includegraphics[width=3.5in]{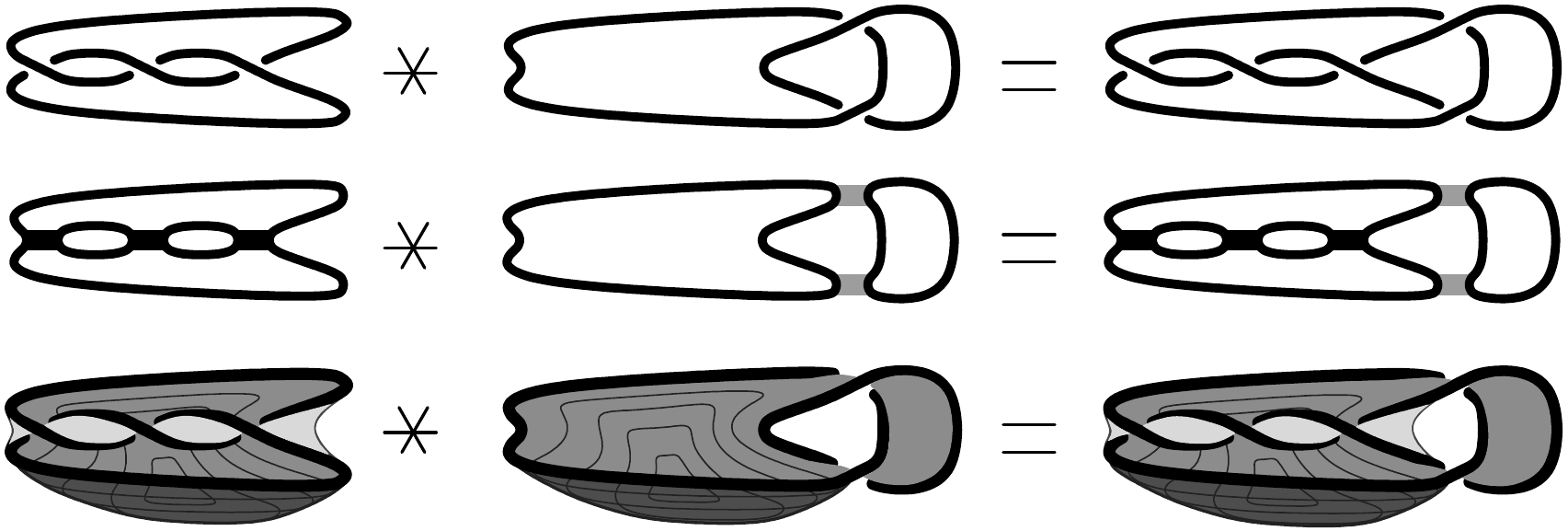}
\caption{Plumbings of link diagrams, states, and surfaces.}\label{Fi:52Plumbings}
\end{center}
\end{figure}

Suppose that $z=x\overset{{g}}{*}y$ is a plumbing of states along a state circle $z_{t_0}={g}(x_{r_0})={g}(y_{s_0})$, and that $p_z={g}(p_x)={g}(p_y)$ is a point away from crossing arcs, so that $p_x\in x_{r_0}$, $p_y\in y_{s_0}$, $p_z\in z_{t_0}$. If $X$, $Y$ enhance $x$, $y$, respectively, in such a way that $p_x$ has the same label in $X$ that $p_y$ has in $Y$, then there is an enhancement $Z$ of $z=x\overset{{g}}{*}y$ which assigns each state circle ${g}(x_{r})\subset {g}(x)$ the same label that $X$ assigns $x_{r}$, and which  assigns each state circle ${g}(y_s)\subset {g}(y)$ the same label that $Y$ assigns $y_s$; denote this enhancement by $Z=X\overset{{g}}{*}Y$ and call it the (external) {plumbing of the chains} $X$, $Y$ by ${g}$.  Extend $R$--linearly to view plumbing 
 as an $R$--module isomorphism:
\begin{align*}
\overset{{g}}{*}:\left(\Cs_{R,\color{ForestGreen}p_x\to1\color{black}}(x)\otimes\Cs_{R,\color{ForestGreen}p_y\to1\color{black}}(y)\right)\oplus\left(\Cs_{R,\color{NavyBlue}p_x\to0\color{black}}(x)\otimes\Cs_{R,\color{NavyBlue}p_y\to0\color{black}}(y)\right)
&\to\Cs_{R}(x\overset{{g}}{*}y),\\
\overset{{g}}{*}:X\otimes Y&\mapsto X\overset{{g}}{*}Y.\\
\end{align*}
There is also an internal notion of plumbing on chains, in which $x$ and $y$ are viewed as subsets of $z$.  For convenience, we take a mixed approach, viewing $x$ and $y$ as subsets of $z$ for simplicity, and using the gluing map $g:(S^2,x)\sqcup(S^2,y)\to(S^2,z)$ for relabeling.  This should cause no confusion.  
How does plumbing of chains, $*$, interact with the differential, ${d}$? Consider the simplest case, connect sum ${\natural}$.
\subsection{Connect sum of chains}\label{S:connectsum} 
A state $z$ is a {\it connect sum} if there is a simple closed curve $\gamma\subset S^2$ which intersects $z$ transversally in two points, $p$, $b$, neither of them on crossing arcs (this implies that $p$, $b$ lie on the same state circle, $z_{t}$), such that both components of $S^2\setminus \gamma$ contain crossing arcs.  In this case, $z$ decomposes as the {connect sum} $x{\natural}y=z$, where $x\subset z$ consists $z_t$ together with all state circles and crossing arcs to one side of $\gamma$ in $S^2$, and $y\subset z$ consists of $z_t$ together with all state circles and crossing arcs to the opposite side of $\gamma$. 
If $D_x$, $D_y$, $D_z$ are the underlying link diagrams for $x$, $y$, $z$ then there is also a connect sum of link diagrams, $D_x{\natural}D_y=D_z$.  Moreover, every state $z'$ of $D_z$ decomposes as a connect sum $x'{\natural}y'=z'$, where $x'$ is a state of $D_x$ and $y'$ is a state of $D_y$.
If $Z'$ enhances $x'{\natural}y'=z'$, then $Z'$ restricts on $x'$, $y'$ to enhancements $X'$, $Y'$, whose labels match at $p$: 
either both are $\color{ForestGreen}1\color{black}$ or both are $\color{NavyBlue}0\color{black}$.  This supplies the $R$--module isomorphism:
\begin{align}\label{E:connectsumiso} \begin{split}
 \Cs_R(D_x{\natural} D_y) &=
 \Cs_{R,\color{ForestGreen}p\to1\color{black}}(D_x{\natural} D_y)\oplus\Cs_{R,\color{NavyBlue}p\to0\color{black}}(D_x{\natural} D_y)\\
&\to
  \left(\Cs_{R,\color{ForestGreen}p\to1\color{black}}(D_x)\otimes \Cs_{R,\color{ForestGreen}p\to1\color{black}}(D_y)\right)\oplus\left(\Cs_{R,\color{NavyBlue}p\to0\color{black}}(D_x)\otimes \Cs_{R,\color{NavyBlue}p\to0\color{black}}(D_y)\right). 
  \end{split}\end{align}
 %
Use this isomorphism to write each enhanced state $Z$ from $D_z=D_x{\natural}D_y$ as $Z=X\natural Y$. If  $D_z=D_x{\natural}D_y$ respects orientation, then $i_{Z}=i_{X}+i_{{Y}}$, and  $j_{Z}=j_{X}+j_{{Y}}+1$ in case ${X}\in\Cs_{R,\color{ForestGreen}p\to 1\color{black}}(D_x)$, ${Y}\in \Cs_{R,\color{ForestGreen}p\to 1\color{black}}(D_y)$ or  $j_{Z}=j_{X}+j_{{Y}}-1$ in case ${X}\in\Cs_{R,\color{NavyBlue}p\to 0\color{black}}(D_x)$, ${Y}\in \Cs_{R,\color{NavyBlue}p\to 0\color{black}}(D_y)$.
%
   
If $X\in \Cs_{\F_2,\color{ForestGreen}p\to1\color{black}\color{black}}\left(D_x\right)$, $Y\in\Cs_{\F_2,\color{ForestGreen}p\to1\color{black}\color{black}}\left(D_y\right)$, then ${d}\left(X{\natural}Y\right)={d} X{\natural}Y+X{\natural}{d} Y.$
This  follows straight from the definition of $d$ (cf Figure \ref{Fi:incidence}). The case $X\in \Cs_{\F_2,\color{NavyBlue}p\to0\color{black}\color{black}}\left(D_x\right)$, $Y\in\Cs_{\F_2,\color{NavyBlue}p\to0\color{black}\color{black}}\left(D_y\right)$ is more awkward, requiring variants of the operation ${\natural}$ in left-- and right--trumps, $\tensor[_\diamondsuit]{{\natural}}{},~\tensor[]{{\natural}}{_\diamondsuit}:\Cs_{\F_2}\left(D_x\right)\otimes \Cs_{\F_2}\left(D_y\right)\to\Cs_{\F_2}\left(D_x{\natural}D_y\right)$. 
%

Suppose that $x{{\natural}}y=z$ is a connect sum of states along a state circle $z_{t_0}=x\cap y$, with $z_{t_0}\ni p$, 
and suppose that $X$, $Y$ enhance $x$, $y$, not necessarily with matching labels at $p$.
Define
$X\tensor[_\diamondsuit]{{\natural}}{}Y$
 to be the enhancement of $x\natural y$ 
which assigns each state circle of $x$ the same label that $X$ assigns it, and which  assigns each state circle  of $y$, \textit{except possibly} $z_{t_0}$, the same label that $Y$ does.  
Likewise, define 
$X\tensor[]{{\natural}}{_\diamondsuit}Y$  
to be the enhancement of $z$ 
which assigns each state circle of $x$ the same label that $X$ does, and which  assigns each state circle  of $y$, \textit{except possibly} $z_{t_0}$, the same label that $Y$ does.  
Thus, 
$X\tensor[_\diamondsuit]{{\natural}}{}Y$
and
$X\tensor[]{{\natural}}{_\diamondsuit}Y$ 
are the enhancements of $z$ which match $X$ and $Y$ away from $z_{t_0}$; at $z_{t_0}$, the label on $z_{t_0}$ from $X$ {\it trumps} the label from $Y$ in $X\tensor[_\diamondsuit]{{\natural}}{}Y$, and vice-versa in $X\tensor[]{{\natural}}{_\diamondsuit}Y$. Whenever $X{\natural}Y$ is defined, it equals both $X\tensor[_\diamondsuit]{{\natural}}{}Y$
and
$X\tensor[]{{\natural}}{_\diamondsuit}Y$.
The most immediate payoff is the following identity, 
which holds over any $R$ with $1$:
 \[{d}(X{\natural}Y)={d} X\tensor[_\diamondsuit]{{\natural}}{}Y+\left(-1\right)^{| \raisebox{-.02in}{\includegraphics[width=.1in]{ASmooth.pdf}}|_X} X\tensor[]{{\natural}}{_\diamondsuit}{d} Y,\]
and in particular over $R=\F_2$:
 \[{d}(X{\natural}Y)={d} X\tensor[_\diamondsuit]{{\natural}}{}Y+X\tensor[]{{\natural}}{_\diamondsuit}{d} Y.\]
%
%
\subsection{General construction.}\label{S:construction}
Let $x*y=z$ be a plumbing of states by a gluing map ${g}:x\sqcup y\to z$, and let $D_x*D_y=D_z$ be the associated plumbing of link diagrams. Index the crossings $c^t_z$ of $D_z$ so that those from $D_x$ precede those from $D_y$: 
$c_z^t=c_x^t$ for $1\leq t\leq \left|\raisebox{-.02in}{\includegraphics[width=.125in]{Crossing.pdf}}\right|_x$, 
$c_z^t=c_y^{t-|\raisebox{-.02in}{\includegraphics[width=.1in]{Crossing.pdf}}|_x}$ for $1+\left|\raisebox{-.02in}{\includegraphics[width=.125in]{Crossing.pdf}}\right|_x\leq t\leq \left|\raisebox{-.02in}{\includegraphics[width=.125in]{Crossing.pdf}}\right|_x+ \left|\raisebox{-.02in}{\includegraphics[width=.125in]{Crossing.pdf}}\right|_y$. Likewise, index the state circles $z_r$ of $z$ so that those from $x$ precede those from $y$: 
$z_r=x_r$ for $1\leq r\leq \left|\bigcirc\right|_x$,  
$z_r=y_{r+1-\left|\bigcirc\right|_x}$ for $\left|\bigcirc\right|_x\leq r\leq \left|\bigcirc\right|_x+ \left|\bigcirc\right|_y-1$. Note that $z_{|\bigcirc|_x}=x_{|\bigcirc|_x}=y_1$.

Let  $X$, $Y$ enhance $x$, $y$, and write $X=q^{a_1}\otimes\cdots\otimes q^{a_{|\bigcirc|_x}}$, $Y=q^{b_1}\otimes\cdots\otimes q^{b_{|\bigcirc|_{y}}}$, with each $a_r,b_r\in\{\color{NavyBlue}0\color{black},\color{ForestGreen}1\color{black}\}$ according to whether the associated state circle is labeled \color{NavyBlue}0 \color{black} or \color{ForestGreen}1\color{black}, as in \textsection\ref{S:viro}.  The plumbing of $X$ and $Y$ by ${g}$ is the enhancement $X*Y=Z$ of the state $x*y=z$ which matches $X$ on the state circles from $x$ and which matches $Y$ on those from $y$, if such an enhancement exists:
\[X
{*}Y=\begin{cases}
q^{a_1}\otimes\cdots \otimes q^{a_{|\bigcirc|_x-1}}\otimes q^{b_1}\otimes q^{b_2}\otimes\cdots\otimes q^{b_{|\bigcirc|_{y}}}&\text{if }
a_{|\bigcirc|_x}= b_1,\\
\text{undefined}&\text{if }
a_{|\bigcirc|_x}\neq b_1.\\ 
\end{cases}\]
Extend $R$--linearly to obtain the following isomorphism of $R$--modules:
\begin{align*}
\overset{{g}}{*}:\left(\Cs_{R,\color{ForestGreen}p\to1\color{black}}(x)\otimes \Cs_{R,\color{ForestGreen}p\to1\color{black}}(y)\right)\oplus\left(\Cs_{R,\color{NavyBlue}p\to0\color{black}}(x)\otimes \Cs_{R,\color{NavyBlue}p\to0\color{black}}(y)\right)&\to \Cs_R(x*y),\\
X\otimes Y&\mapsto X*Y.
\end{align*}
Recall from \textsection\ref{S:intro} that Khovanov homology over $\F_2$ detects adequate all--$\color{gray}A\color{black}$ and all--$B$ states, in the sense that such a state has enhancements $X^\pm$ with $j_{X^+}=j_{X^-}+2$, such that $\text{tr}_{\F_2}X^+$, $\text{tr}_{\F_2}X^-$ represent nonzero homology classes.  Moreover, every homogeneously adequate state is a plumbing of adequate all--$\color{gray}A\color{black}$ and all--$B$ states.  
The main theorem will follow inductively from this setup, using the interaction between plumbing and the differential, which we describe next. 
\subsection{Trumps.}\label{S:trumps}
Let $x*y=z$ be a plumbing by a gluing map $g$, 
and let $x'$, $y'$ be arbitrary states of $D_x$, $D_y$.  In an abuse of terminology and notation, define the {plumbing}  $x'* y$ to be the state of $D_x*D_y=D_z$ whose smoothings match those of $x'$ and $y$; likewise, define the {plumbing}  $x* y'$ to be the state of $D_z$ whose smoothings match those of $x$ and $y'$.  In terms of the crossing ball setup from \textsection\ref{S:statesurface}, 
\[x'{*}{y}=\big((x'\cup y)\setminus \left(x\cap y\right)\big)\cup \big(x'\cap y\big),\qquad
x{*}{y'}=\big((x\cup y')\setminus \left(x\cap y\right)\big)\cup \big(x\cap y'\big).
\]%
%
If $X'$ is an enhancement of $x'$ and $Y$ is an enhancement of $y$, define the {\it left--trump plumbing $X'\tensor[_\diamondsuit]{{}{*}}{}Y$} by $g$ to be the enhancement of $x'*y$  which assigns each state circle $x'_r\subset x'$ the same label that $X'$ assigns $x'_r$, and which  assigns each state circle $y_r\subset y$, \textit{except possibly} $y_1$ which need not be a state circle in $x'*y$, the same label that $Y$ assigns $y_r$.
Likewise, 
if $Y'$ is an enhancement of a state $y'$ of $D_y$ and $X$ is an enhancement of $x$, the {\it right--trump plumbing $X\tensor[]{{}{*}}{_\diamondsuit}Y'$} is the enhancement of $x*y'$  which assigns each state circle $x_r\subset x$, \textit{except possibly} $x_{|\bigcirc|_x}$, the same label that that $X$ does, and which  assigns each state circle $y'_r\subset y'$ the same label that $Y'$ does. That is, 
$X'\tensor[_\diamondsuit]{{\natural}}{}Y$
and
$X\tensor[]{{\natural}}{_\diamondsuit}Y'$ 
are the respective enhancements of $x'*y$ and $x*y'$ which match $X'$, $Y$ and $X$, $Y'$ away from ${g}(x)\cap {g}(y)=z_{|\bigcirc|_x}$; at $z_{|\bigcirc|_x}$, the labels from $X'$ {\it trump} the labels from $Y$ in $X'\tensor[_\diamondsuit]{{\natural}}{}Y$, and the labels from $Y'$ {\it trump} those from $X$ in  $X\tensor[]{{\natural}}{_\diamondsuit}Y'$:
\begin{align*}\tensor[_\diamondsuit]{\overset{{g}}{*}}{}:\Cs_{R}(D_x)\otimes \Cs_{R}(y)&\to\Cs_{R}\left(D_x{\natural}D_y\right),
\hspace{.5in}&
\tensor[]{\overset{{g}}{*}}{_\diamondsuit}:\Cs_{R}(x)\otimes \Cs_{R}(D_y)&\to\Cs_{R}\left(D_x{\natural}D_y\right),\\
X'\otimes Y&\mapsto X'\tensor[_\diamondsuit]{\overset{{g}}{*}}{}Y&X\otimes Y'&\mapsto X\tensor[]{\overset{{g}}{*}}{_\diamondsuit}Y'\\, \end{align*}
\begin{prop}\label{P:trumpPlumb}
If $X{*}Y$ enhances $x* y=z$, then ${d}(X*Y)={d} X \tensor[_\diamondsuit]{*}{}Y+\left(-1\right)^{| \raisebox{-.02in}{\includegraphics[width=.1in]{ASmooth.pdf}}|_x}
 X\tensor[]{*}{_\diamondsuit}{d} Y$.\end{prop}
\begin{proof}Let $D_x$, $D_y$, $D_z$ be the 
link diagrams for $x$, $y$, $z$. Index the crossings as in \textsection\ref{S:construction}, and again let $| \raisebox{-.02in}{\includegraphics[width=.12in]{ASmooth.pdf}}|_{x}^t$ denote the number of crossings $c_x^r$ in $D_x$ with $r<t$ at which $x$ has an $\color{gray}A\color{black}$--smoothing. 
Now:
\begin{align*}\pushQED{\qed} {d} \big(&{X{*}Y}\big)
=\sum_{t=1}^{|\raisebox{-.02in}{\includegraphics[height=.1in]{Crossing.pdf}}|_{D_x}}
\left(-1\right)^{| \raisebox{-.02in}{\includegraphics[width=.1in]{ASmooth.pdf}}|_{z}^t}
{d}_{c_z^t}( {X{*}Y})
+ \sum_{t=1+|\raisebox{-.02in}{\includegraphics[height=.1in]{Crossing.pdf}}|_{D_x}}
^{|\raisebox{-.02in}{\includegraphics[height=.1in]{Crossing.pdf}}|_{D_x}+|\raisebox{-.02in}{\includegraphics[height=.1in]{Crossing.pdf}}|_{D_y}}
\left(-1\right)^{| \raisebox{-.02in}{\includegraphics[width=.1in]{ASmooth.pdf}}|_{z}^t}
{d}_{c_z^t}\big({X{*}Y}\big)\\
&=\sum_{t=1}^{|\raisebox{-.02in}{\includegraphics[height=.1in]{Crossing.pdf}}|_{D_x}} 
\left(-1\right)^{| \raisebox{-.02in}{\includegraphics[width=.1in]{ASmooth.pdf}}|_{x}^t}
{d}_{c_x^t}X\tensor[_\diamondsuit]{*}{}Y
+\left(-1\right)^{| \raisebox{-.02in}{\includegraphics[width=.1in]{ASmooth.pdf}}|_x}
\sum_{t=1}^{|\raisebox{-.02in}{\includegraphics[height=.1in]{Crossing.pdf}}|_{D_y}} 
\left(-1\right)^{| \raisebox{-.02in}{\includegraphics[width=.1in]{ASmooth.pdf}}|_{y}^t}
X\tensor[]{*}{_\diamondsuit}{d}_{c_y^t} Y\\
&={d} X\tensor[_\diamondsuit]{*}{}Y
+\left(-1\right)^{| \raisebox{-.02in}{\includegraphics[width=.1in]{ASmooth.pdf}}|_x} X\tensor[]{*}{_\diamondsuit}{d} Y.
\qedhere \\ \end{align*}
\end{proof}
%
%
\begin{figure}
\begin{center}
\scalebox{2}{$\raisebox{-.05in}{\includegraphics[width=.4in]{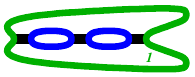}}*\raisebox{-.05in}{\includegraphics[width=.55in]{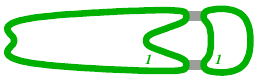}}=\raisebox{-.05in}{\includegraphics[width=.55in]{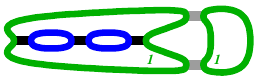}}$}
\caption{If $x{*}y$ is a plumbing of states and $X\in \Cs_R(x)$, $Y\in\Cs_R(y)$ are cycles with $X*Y\in \Cs_R(x*y)$, then $X*Y$  is also a cycle (cf Proposition \ref{P:trumpPlumb}).}\label{Fi:CyclePlumbing}
\end{center}
\end{figure}
\subsection{Cycles}
Let $\kappa:(S^2,x)\sqcup({S^2},\bigcirc)\to (S^2,x)$ be a gluing map of a state $x$ with the state $\bigcirc$ of the trivial diagram,  
so that $x=x\overset{_{\kappa}}{*}\bigcirc$ and $\kappa|_x=\id_x$. 
Suppose $p$ is a point on the state circle $\kappa^{-1}(\bigcirc)\subset x$, and
$X\in \Cs_{R,\color{ForestGreen}p\to1\color{black}}(x)$, $X'\in \Cs_{R,\color{NavyBlue}p\to0\color{black}}(x)$ are chains. 
Say that $X$ and $X'$ are {\it identical away from $p$} 
if $X\,{\tensor[]{\overset{\kappa}{*}}{_\diamondsuit}}\color{NavyBlue}\bigcirc\color{black}=X'$, or equivalently if $X'\,{\tensor[]{\overset{\kappa}{*}}{_\diamondsuit}}\color{ForestGreen}\bigcirc\color{black}=X$. 
\begin{obs}\label{O:identical}
If $X,X'\in \Cs_R(x)$ are identical away from $p$, 
and if $x*y=z$ is a plumbing of states by a gluing map $g:x\sqcup y\to z$ with $x\cap g^{-1}(y)=x\cap \kappa^{-1}(\bigcirc)\ni p$,
then $X\tensor[]{*}{_\diamondsuit}Y=X'\tensor[]{*}{_\diamondsuit}Y$  for any $Y\in \Cs_R(y)$.
\end{obs}
\begin{prop}\label{P:plumbCycle}
If $x{*}y$ is a plumbing of states and $X$, $X'$, $Y+Y'$  are cycles, with $X\in\Cs_{R,\color{ForestGreen}p\to1\color{black}}(x)$, $X'\in\Cs_{R,\color{NavyBlue}p\to0\color{black}}(x)$ identical away from $p\in x\cap y$ and $Y\in\Cs_{R,\color{ForestGreen}p\to1\color{black}}(y)$, $Y'\in\Cs_{R,\color{NavyBlue}p\to0\color{black}}(y)$, then $X*Y+X'*Y'$  is also a cycle. 
\end{prop}
\begin{proof}
Since $X$ and $X'$ are identical away from $p$, Observation \ref{O:identical} implies that $X\tensor[]{*}{_\diamondsuit}dY'=X'\tensor[]{*}{_\diamondsuit}dY'$.  
This and Proposition \ref{P:trumpPlumb} now yield:
\begin{align*}
\pushQED{\qed} {d} \big({X{*}Y}+{X'{*}Y'}\big)
&={d} X \tensor[_\diamondsuit]{*}{}Y+{d} X' \tensor[_\diamondsuit]{*}{}Y'+\left(-1\right)^{| \raisebox{-.02in}{\includegraphics[width=.1in]{ASmooth.pdf}}|_x}
\left( X\tensor[]{*}{_\diamondsuit}{d} Y+X'\tensor[]{*}{_\diamondsuit}{d} Y'\right)\\
&=\left(-1\right)^{| \raisebox{-.02in}{\includegraphics[width=.1in]{ASmooth.pdf}}|_x}
\left(X\tensor[]{*}{_\diamondsuit} {d} Y+X\tensor[]{*}{_\diamondsuit} {d} Y'\right)\\
&=\left(-1\right)^{| \raisebox{-.02in}{\includegraphics[width=.1in]{ASmooth.pdf}}|_x}
\left(X\tensor[]{*}{_\diamondsuit} {d} (Y+Y')\right)=0. \qedhere \\
\end{align*}
\end{proof}
\begin{figure}
\begin{center}
\includegraphics[width=6.5in]{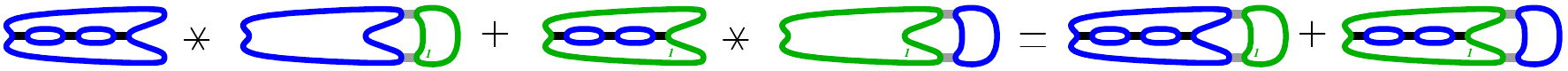}
\caption{
If $x{*}y$ is a plumbing of states and $X$, $X'$, $Y+Y'$  are cycles, with $X,X'\in \Cs_R(x)$ identical away from $x\cap y$ and  $X*Y+X'*Y'\in \Cs_R(x*y)$, then $X*Y+X'*Y'$  is also a cycle (cf Proposition \ref{P:plumbCycle}). 
}
\label{Fi:InductivePlumbing}
\end{center}
\end{figure}
\begin{obs}\label{O:identicalB}
If $x*\bigcirc =x$ is a plumbing by $\kappa$ as above, 
with $\kappa^{-1}(\bigcirc)\cap x_{\color{gray}A\color{black}}=\varnothing$ and $X\in \Cs_{R}(x)$, then:
\[{d} (X\tensor[]{*}{_\diamondsuit}\color{ForestGreen}\bigcirc\color{black})={d} X\tensor[]{*}{_\diamondsuit}\color{ForestGreen}\bigcirc\color{black}\hspace{1in}\text{and}\hspace{1in}{d} (X\tensor[]{*}{_\diamondsuit}\color{NavyBlue}\bigcirc\color{black})={d} X\tensor[]{*}{_\diamondsuit}\color{NavyBlue}\bigcirc\color{black}.\]
In particular, if $X$ is a cycle with $\kappa^{-1}(\bigcirc)\cap x_{\color{gray}A\color{black}}=\varnothing$, then $X\tensor[]{*}{_\diamondsuit}\color{ForestGreen}\bigcirc\color{black}$, $X\tensor[]{*}{_\diamondsuit}\color{NavyBlue}\bigcirc\color{black}$ are also cycles.  
\end{obs}
The point is that, because the state circle $\kappa^{-1}(\bigcirc)$ is incident to no $\color{gray}A\color{black}$--type crossing arcs, every enhanced state $Y$ with $ \pi_{RY}\circ {d} X\neq 0$ contains $\kappa^{-1}(\bigcirc)$ and assigns it the same label that $X$ does. 
\begin{prop}\label{P:plumbCycleGen}
Let $z=x
{*}y$ be a plumbing of states such that 
$x\cap y\subset y\setminus y_{\color{gray}A\color{black}}$.
If $X$, $Y$ enhance $x$, $y$ such that both $\text{tr}_R{X}$ and $\text{tr}_R{Y}$ are cycles and $X*Y$ is defined, then $\text{tr}_R(X*Y)$ is also a cycle.
\end{prop}
\begin{proof}
Use the indexing from \textsection\ref{S:construction}, so that the state circles in $z$ from $x$ precede those from $y$, with $x\cap y=x_{|\bigcirc|_x}=y_1$.  
The cycle condition on $\text{tr}_R{X}$, $\text{tr}_R{Y}$ implies that each component of $x_{\color{gray}A\color{black}}$, $y_{\color{gray}A\color{black}}$ is adequate and contains at most one \color{NavyBlue}0\color{black}--labeled circle, by Observation \ref{O:generic}.  Write $X=\bigotimes_{r=1}^{|\bigcirc|_x} q^{a_r}$ and $Y=\bigotimes_{r=1}^{|\bigcirc|_{y}} q^{b_r}$. 
Let
$Y':=
1\otimes \bigotimes_{r=2}^{|\bigcirc|_{y}} q^{b_r}$ and 
$Y'':=q\otimes \bigotimes_{r=2}^{|\bigcirc|_{y}} q^{b_r}$, so that $Y'$, $Y''$ are identical away from $y_1$, and one of $Y'$, $Y''$ equals $Y$.  Thus, one of 
 one of 
$\text{tr}_R{Y'}$,  $\text{tr}_R{Y''}$ 
equals $\text{tr}_R{Y}$, which is a cycle; the assumption that $y_1\subset y\setminus y_{\color{gray}A\color{black}}$ implies that the other is also a cycle, by Observation \ref{O:identicalB}. 
 Write $\text{tr}_R{X}=X'+X''$, where $X'\in \Cs_{R,\color{ForestGreen}p\to1\color{black}}(x)$ and  $X''\in \Cs_{R,\color{NavyBlue}p\to0\color{black}}(x)$ with $p\in x_{|\bigcirc|_x}$. 
 Proposition \ref{P:plumbCycle} now implies that $\text{tr}_R(X*Y)=X'{*} \text{tr}_R{Y'}+{X''}{*} \text{tr}_R{Y''}$ is a cycle.
\end{proof}
\subsection{Boundaries}
Consider the following chains from Figure \ref{Fi:TrefGridJones}:
\[\raisebox{-.1in}{\includegraphics[height=.3in]{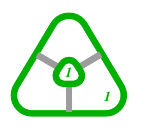}},\hspace{.25in}
\raisebox{-.1in}{\includegraphics[height=.3in]{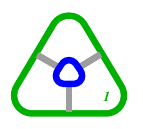}}-\raisebox{-.1in}{\includegraphics[height=.3in]{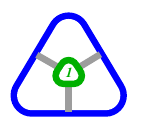}},\hspace{.25in} 
 \raisebox{-.1in}{\includegraphics[height=.3in]{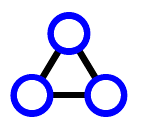}},\hspace{.25in}
 \raisebox{-.1in}{\includegraphics[height=.3in]{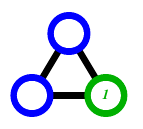}}.\]
All four are closed; are they exact? The first three cannot be exact since their $B$--type crossing arcs, if there are any, join distinct \color{NavyBlue}0\color{black}--labeled circles; this holds over both $R=\F_2,\Z$. 
 To see that $X:=\raisebox{-.06in}{\includegraphics[height=.2in]{Tref11q.pdf}}$ is not exact over $\F_2$, $\Z$, apply Proposition \ref{P:1} and the homogeneity of $x:=\raisebox{-.06in}{\includegraphics[height=.2in]{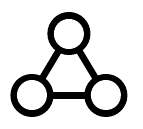}}$ 
to see that
$\left[X\right]_{\color{gray}A\color{black}}\cap\left[X\right]_B=\{X\}$. 
Since each $B$--type crossing arc in $X$ (and in its two $B$--equivalent enhanced states) joins distinct state circles, at most one of them labeled \color{ForestGreen}1\color{black}, the image of the map $\varepsilon\circ\pi_{[X]_B}\circ {d}$ is in $2R$.   This implies that $X$ cannot be exact:
 \begin{prop}\label{P:inexact}
 If $X$ enhances a state $x$ of a diagram $D$ such that $[X]_A\cap [X]_B=\{X\}$ (eg if $x$ is homogeneous), and if $\varepsilon\circ\pi_{[X]_B}\circ d:\Cs_R(D)\to 2R$ with $2R\subsetneqq R$, then $\text{tr}_RX$ is not exact.
 \end{prop}
 \begin{proof}
If $\text{tr}_RX$ were exact, say $\text{tr}_RX= {d} Y$, $Y\in \Cs_R(D)$, then 2 would be a unit in $R$, contrary to assumption:
\[\pushQED{\qed}1=\varepsilon(X)= \varepsilon\circ\pi_{[X]_B}(\text{tr}_RX) =\varepsilon\circ\pi_{[X]_B}\circ{d} (Y)\in 2R.\\ \qedhere \]
\end{proof}
Thus, $\text{tr}_RX=X=\raisebox{-.06in}{\includegraphics[height=.2in]{Tref11q.pdf}}$ is not exact because $[X]_A\cap [X]_B=\{X\}$ and $\varepsilon\circ\pi_{[X]_B}\circ {d}:\Cs_R(D)\to 2R$.  
Plumbing preserves homogeneity, which implies the former property.  Plumbing also preserves the latter property:
 \begin{prop}\label{P:plumb2R}
If $X*Y$ is a plumbing of chains, where $X$, $Y$ enhance $x$, $y$ with $\varepsilon\circ\pi_{[X]_B}\circ {d}: \Cs_R(D_x)\to 2R$, $\varepsilon\circ\pi_{[Y]_B}\circ {d}: \Cs_R(D_y)\to 2R$, and if $x*y=z$ is a plumbing of states, then there are $X'\in [X]_B$, $Y'\in [Y]_B$ such that $X'*Y'=:Z$ satisfies  $\varepsilon\circ\pi_{[Z]_B}\circ {d}: \Cs_R(D_z)\to 2R$.
 \end{prop}
 \begin{proof}
Deduce from $\varepsilon\circ\pi_{[X]_B}\circ {d}: \Cs_R(D_x)\to 2R$ that $X$ has no more than one 
\color{ForestGreen}1\color{black}--labeled circle in any
 $B$--zone of $x$ (cf Figure \ref{Fi:incidence}); likewise for $Y$, $y$.
Let $z_0$ denote the state circle $x\cap y$, and let $p$ be a point on $z_0$ away from crossings.  Let $X'=X$ and $Y'=Y$, {\it unless}
$z_0$ is in $B$--zones of both $x$, $y$ such that both zones contain a $\color{ForestGreen}1\color{black}$--labeled circle. In that case, 
choose $X'\in [X]_B\cap \Cs_{R,\color{ForestGreen}p\to1\color{black}}(x)$ and $Y'\in [Y]_B\cap \Cs_{R,\color{ForestGreen}p\to1\color{black}}(y)$, so that $X'*Y'=:Z\in \Cs_{R,\color{ForestGreen}p\to1\color{black}}(z)$.  
In all cases, we have chosen $X'$, $Y'$ so that, if $z_0$ is in a $B$--zone of $z$, then this zone contains at most one $\color{ForestGreen}1\color{black}$--labeled circle in $Z$. 

We claim that these choices for $X'*Y'=:Z$ always satisfy  $\varepsilon\circ\pi_{[Z]_B}\circ {d}: \Cs_R(D_z)\to 2R$.
Suppose $W$ enhances a state $w$ of $D_z$ such that $\pi_{[Z]_B}\circ {d}(W)\neq 0$.  The state $w$ must differ from $z$ at a single crossing. Either this crossing is from $y$, and $w=x* y'$ for a state $y'$ of $D_y$; or the crossing is from $x$, and
$w=x'* y$ for a state $x'$ of $D_x$; \textsc{wlog} assume the latter.  Then, since the $B$--zone of $y$ containing $z_0$, if there is one, has no $\color{ForestGreen}1\color{black}$--labeled circles other than $z_0$, $W$ assigns each circle of $y\setminus z_0$ the same label that $Y'$ does; ie $W=X''\tensor[_\diamondsuit]{{*}}{}Y'$ for some enhancement $X''$ of $x'$. Thus, $
\pi_{[Z]_B}\circ {d}(X''\tensor[_\diamondsuit]{{*}}{}Y')=\pi_{[Z]_B}\left( {dX''}\tensor[_\diamondsuit]{{*}}{}Y'\right)$, giving:
\[\pushQED{\qed}\varepsilon\circ\pi_{[Z]_B}\circ {d}(W)
=\varepsilon\circ\pi_{[Z]_B}\left( {dX''}\tensor[_\diamondsuit]{{*}}{}Y'\right)
=
\varepsilon\circ\pi_{[X]_B}\circ d(X'')
\in 2R.\qedhere\]
 \end{proof}
\subsection{Inductive proof of the main theorem}
Two examples will show how plumbing is used to build up the main theorem's nonzero cycles.  First, with either $R=\F_2$ or $R=\Z$, consider: 
\[X_1=\raisebox{-15pt}{\includegraphics[height=35pt]{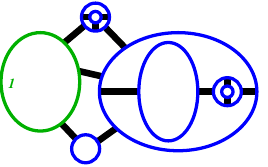}}
, ~X_2=\raisebox{-9pt}{\includegraphics[height=25pt]{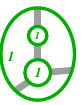}}
, ~X_3=\raisebox{-9pt}{\includegraphics[height=23pt]{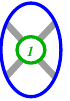}}
.\]
Each of $\text{tr}_R X_1=X_1$, $\text{tr}_R X_2=X_2$, $\text{tr}_R X_3-\raisebox{-6pt}{\includegraphics[height=18pt]{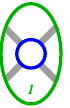}}
$
is a cycle; also each $[X_r]_A\cap[X_r]_B=\{X_r\}$,  
and $\varepsilon\circ\pi_{[X_r]_B}\circ d$ maps to $2R$; Proposition \ref{P:inexact} implies that 
$\text{tr}_R X_1$, $\text{tr}_R X_2$, $\text{tr}_R X_3$ represent nonzero homology classes. 
Propositions \ref{P:plumbCycleGen}, \ref{P:plumb2R} further imply that
\[\text{tr}_R(X_1*X_2)=X_1*X_2=\raisebox{-15pt}{\includegraphics[height=35pt]{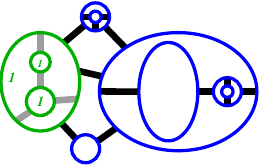}}\]
also represents a nonzero homology class, as does
\[\text{tr}_R\left(X_1*X_2*X_3\right)=
\raisebox{-15pt}{\includegraphics[height=35pt]{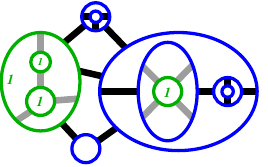}}-
\raisebox{-15pt}{\includegraphics[height=35pt]{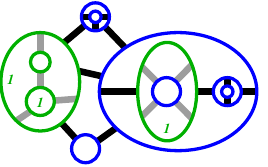}}.\]
While the previous example holds over both $\Z$, $\F_2$, the next example works over $\F_2$ only. Let
\[Y_1=\raisebox{-15pt}{\includegraphics[height=35pt]{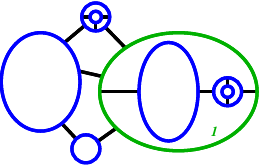}}
, ~Y_2=\raisebox{-9pt}{\includegraphics[height=25pt]{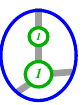}}
, ~Y_3=\raisebox{-9pt}{\includegraphics[height=23pt]{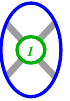}}
.\]
By the same reasoning as the last example,
$\text{tr}_{\F_2} Y_1=Y_1$, $\text{tr}_{\F_2} Y_2=Y_2+
\raisebox{-6pt}{\includegraphics[height=18pt]{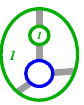}}+
\raisebox{-6pt}{\includegraphics[height=18pt]{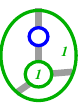}}
$, and 
$\text{tr}_{\F_2} Y_3=Y_3+
\raisebox{-6pt}{\includegraphics[height=18pt]{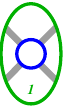}}
$ represent nonzero homology classes, as do
\[\text{tr}_{\F_2}(Y_1*Y_2)=\raisebox{-15pt}{\includegraphics[height=35pt]{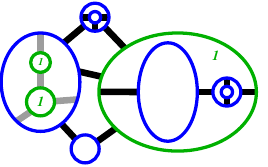}}+
\raisebox{-15pt}{\includegraphics[height=35pt]{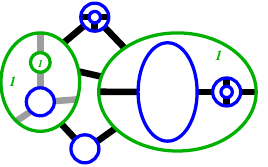}}+
\raisebox{-15pt}{\includegraphics[height=35pt]{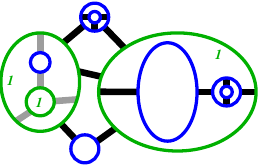}}\]
{ and}
\[\text{tr}_{\F_2}\left(Y_1*Y_2*Y_3\right)=
\raisebox{-15pt}{\includegraphics[height=35pt]{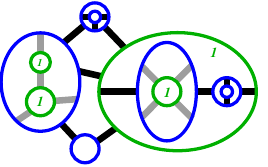}}+
\raisebox{-15pt}{\includegraphics[height=35pt]{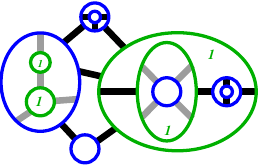}}+
\raisebox{-15pt}{\includegraphics[height=35pt]{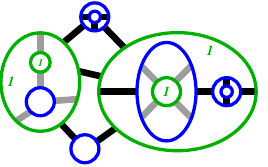}}+
\raisebox{-15pt}{\includegraphics[height=35pt]{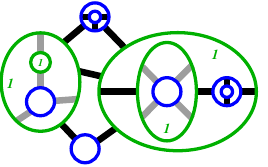}}+
\raisebox{-15pt}{\includegraphics[height=35pt]{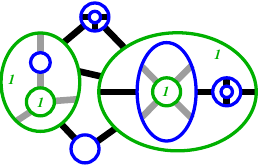}}+
\raisebox{-15pt}{\includegraphics[height=35pt]{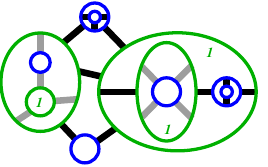}}
.\]

Both proofs of the main theorem will establish the following, which is stronger than the version from \textsection\ref{S:intro}. 
\begin{maintheorem}
If $z$ is a homogeneously adequate state, then for any point $p$ 
away from crossing arcs, $z$ has enhancements 
$Z^-\in \Cs_{R,\color{ForestGreen}p\to1\color{black}}(x)$, $Z^+\in \Cs_{R,\color{NavyBlue}p\to0\color{black}}(z)$,  
identical away from $p$, such that 
both $\text{tr}_{\F_2}{Z^\pm}$ represent nonzero homology classes. 
If also $G_{x_{\color{gray}A\color{black}}}$ is bipartite, then both $\text{tr}_{\Z}Z^\pm$ represent nonzero homology classes.
\end{maintheorem}
\begin{proof}
We argue by induction on the number of zones in $z$ that
$z$ has enhancements $Z^+\in \Cs_{R,\color{NavyBlue}p\to0\color{black}}(z)$, $Z^-\in \Cs_{R,\color{ForestGreen}p\to1\color{black}}(z)$, identical away from 
$p$, such that both {$\color{gray}A\color{black}$--trace}s $\text{tr}_{R}Z^\pm$ are cycles, and $2R$ contains the images of $\varepsilon\circ \pi_{[Z^\pm]_B}\circ {d}$. 
The last condition implies that neither $Z^\pm$ is exact, by Proposition \ref{P:inexact}. 

The base case checks out. For the inductive step, de-plumb $z=x*y$ by a gluing map ${g}:x\sqcup y\to z$ with ${g}(x_{r_0})={g}(y_{s_0})=z_{t_0}$, where either $x_{r_0}\subset x\setminus x_{\color{gray}A\color{black}}$ or $y_{s_0}\subset y\setminus y_{\color{gray}A\color{black}}$. 

If $p\in z_{t_0}$, then apply the inductive hypothesis to $x$ and $y$ to obtain $X^+\in \Cs_{R,\color{NavyBlue}p\to0\color{black}}(x)$, $X^-\in \Cs_{R,\color{ForestGreen}p\to1\color{black}}(x)$, identical away from 
$p$, and $Y^+\in \Cs_{R,\color{NavyBlue}p\to0\color{black}}(y)$, $Y^-\in \Cs_{R,\color{ForestGreen}p\to1\color{black}}(y)$, identical away from 
$p$, such that all four of $\text{tr}_RX^\pm$, $\text{tr}_RY^\pm$ are cycles, 
and such that $2R$ contains the images of all four of
$\varepsilon\circ \pi_{[X^\pm]_B}\circ {d}$, 
$\varepsilon\circ \pi_{[Y^\pm]_B}\circ {d}$.
Let $Z^-:=X^-{*}Y^-$, $Z^+:=X^+{*}Y^+$.  Then $Z^-$, $Z^+$ are identical away from $z_{t_0}$; $\text{tr}_R{Z^-}$, $\text{tr}_R{Z^+}$ are cycles; 
and 
$2R$ contains the images of
$\varepsilon\circ \pi_{[Z^-]_B}\circ {d}$, $\varepsilon\circ \pi_{[Z^+]_B}\circ {d}$.

Assume instead $p\notin z_{t_0}$; \textsc{wlog} $p\in x\setminus y$. 
Apply the inductive hypothesis to obtain $X^+\in \Cs_{R,\color{NavyBlue}p\to0\color{black}}(x)$, $X^-\in \Cs_{R,\color{ForestGreen}p\to1\color{black}}(x)$, identical away from $p$, 
such that $\text{tr}_RX^\pm$ are cycles 
and  $2R$ contains the images of  
$\varepsilon\circ \pi_{[X^\pm]_B}\circ {d}$. 
Also let $b$ be a point in $z_{t_0}$, and apply the inductive hypothesis to obtain $Y^+\in \Cs_{R,\color{NavyBlue}b\to0\color{black}}(y)$, $Y^-\in \Cs_{R,\color{ForestGreen}b\to1\color{black}}(y)$, identical away from $p$
, such that $\text{tr}_RY^\pm$ are cycles 
and  $2R$ contains the images of  
$\varepsilon\circ \pi_{[Y^\pm]_B}\circ {d}$. Since $X^\pm$ are identical away from $p$, 
they assign the same label to $z_{t_0}$. If this label is \color{ForestGreen}1\color{black}, then let $Z^+:=X^+*Y^+$ and $Z^-:=X^-*Y^+$; if it is \color{NavyBlue}0\color{black}, then define $Z^+:=X^+*Y^-$ and $Z^-:=X^-*Y^-$. Either way, $Z^\pm$ are identical away from $p$, $\text{tr}_R{Z^\pm}$ are cycles,
and $2R$ contains the images of  
$\varepsilon\circ \pi_{[Z^\pm]_B}\circ {d}$.
 \end{proof}

\section{Direct proof of the main theorem.}\label{S:direct}
Throughout this section, fix a homogeneously adequate state $x$ of a link diagram $D$. 
Here is the plan.  Several propositions will establish two conditions on enhancements $X$ of $x$ which together guarantee that $\text{tr}_{\F_2}{X}$ represents a nonzero homology class: each $\color{gray}A\color{black}$--zone must contain at most one \color{NavyBlue}0\color{black}--labeled circle, and each $B$--zone must contain at most one \color{ForestGreen}1\color{black}--labeled circle.  These conditions also suffice over $R=\Z$ when $G_{x_{\color{gray}A\color{black}}}$ is bipartite.  An explicit construction will then fashion enhancements $X^\pm$ of $x$ which satisfy these conditions, with $j(X^+)=j(X^-)+2$.   

\begin{prop}\label{P:2}
If $X$ enhances $x$ with at most one \color{NavyBlue}0\color{black}--labeled circle in each $\color{gray}A\color{black}$--zone, then ${d}\left(\text{tr}_{\F_2}{X}\right)=0$.  Further, ${d}\left(\text{tr}_{\Z}{X}\right)=0$ if every non-bipartite $\color{gray}A\color{black}$--zone of $X$ is all--\color{ForestGreen}1\color{black}.
\end{prop}
 \begin{proof}
Let $c$ be an arbitrary crossing of the link diagram ${D}$; it will suffice to show that ${d}_c\left(\text{tr}_{R}{X}\right)=0$ for $R=\F_2$, $R=\Z$. Assume that $x$ has an $\color{gray}A\color{black}$--type crossing arc at $c$, or else we are done. Partition the enhanced states in $[X]_{\color{gray}A\color{black}}$ as follows.  Let one equivalence class consist of all enhancements for which both state circles incident to $c$ are labeled \color{ForestGreen}1\color{black}; ${d}_c(X')=0$ for each $X'$ in this class.  Partition any remaining enhanced states in $[X]_{\color{gray}A\color{black}}$ into pairs $\{X_s,X_{s'}\}$ which are identical except with opposite labels on the two state circles incident to $c$.  For each such pair,  ${d}_c(X_s)={d}_c(X_{s'})$ over both $R=\F_2$ and $R=\Z$; also, $\text{sgn}(X\to X_s)=-\text{sgn}(X\to X_{s'})$ in case $R=\Z$. Conclude in both cases: 
\[
\pushQED{\qed} 
{d}\left(\text{tr}_{\Z}{X}\right)
=\sum_{ X'\in[X]_{\color{gray}A\color{black}}}\text{sgn}(X\to X'){d}_cX'
=\sum_{ \text{pairs }\{X_s,~X_{s'}\}}{\text{sgn}(X\to X_{s})}\left({d}_cX_{s}-{d}_cX_{s'}\right)=0.
 \qedhere 
\]
\end{proof}
\begin{prop}\label{P:3}
If $X$ enhances $x$ 
so that no $B$--zone contains more than one \color{ForestGreen}1\color{black}--labeled circle, then
\[\varepsilon\circ \pi_{\left[X\right]_B}\circ {d}:\Cs_R(D)\to 2R.\]
\end{prop}
\begin{proof}
Let $Y$  
be any enhanced state. If $\pi_{[X]_B}\circ {d}(Y)\neq 0$,
then the underlying state $y$ of $Y$ must differ from $x$ at precisely one crossing, $c$, at which $y$ must have an $\color{gray}A\color{black}$--smoothing with one incident state circle, which must be labeled \color{NavyBlue}0 \color{black} in $Y$ because each $X'\in[X]_B$ has at most one \color{ForestGreen}1\color{black}--labeled circle in each $B$--zone.  Thus, $\pi_{[X]_B}\circ {d}_c(Y)=X_s+X_{s'}$, where $X_s$, $X_{s'}$ are identical except with opposite labels on the two state circles of $x$ incident to $c$.  In particular, $\varepsilon\circ \pi_{[X]_B}\circ {d}(Y_t)=1+1\in2R$.
\end{proof}
\begin{prop}\label{P:4}
If $X$ enhances $x$ with at most one \color{NavyBlue}0\color{black}--labeled state circle in each $\color{gray}A\color{black}$--zone and at most one \color{ForestGreen}1\color{black}--labeled
  state circle in each $B$--zone, then $\text{tr}_{\F_2}{X}$ represents a nonzero homology class. 
Further, if every $\color{gray}A\color{black}$--zone containing a \color{NavyBlue}0\color{black}--labeled circle in $X$ is bipartite, then $\text{tr}_\Z{X}$ represents a nonzero homology class.
\end{prop}
\begin{proof}
Such $\text{tr}_{\F_2}{X}$, $\text{tr}_\Z{X}$ are cycles by Proposition \ref{P:2}. If $\text{tr}_R{X}$ were exact over $R=\F_2$ or $R=\Z$, say $\text{tr}_RX=dY$, then Propositions \ref{P:1} and \ref{P:3} would imply that 2 is a unit in $R$:
 \[\pushQED{\qed}1=\varepsilon(X)=\varepsilon\left(\sum_{X'\in{[X]}_{\color{gray}A\color{black}}\cap{[X]}_B}X'\right)=\varepsilon\circ \pi_{[X]_B}\left(\text{tr}_R{X}\right)=\varepsilon\circ \pi_{[X]_B}\circ {d} Y\in 2R.\qedhere\]
 \end{proof}
Putting all this together proves that Khovanov homology over $\F_2$ detects every homogeneously adequate state $x$ in two distinct gradings, $(i_x,j_x\pm1)$, where 
\[j_x={w_D}+i_x+\left|\bigcirc\right|_{x_B}-\left|\bigcirc\right|_{x_{\color{gray}A\color{black}}}+\#(B\text{--zones of }x)-\#(A\text{--zones of }x).\]
 \begin{figure}
\begin{center}
\includegraphics[width=6.5in]{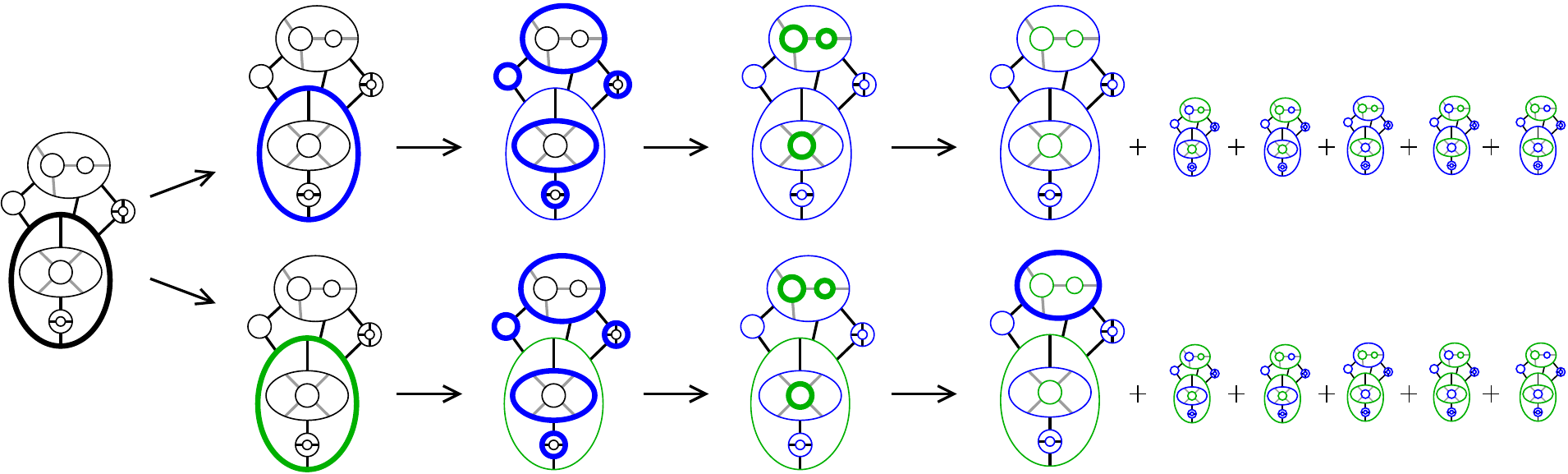}
\caption{Constructing $X^\pm$ from the state $x$, left.}
\label{Fi:Construct}
\end{center}
\end{figure}
\begin{maintheorem}
If $x$ is a homogeneously adequate state, then for any point $p$ 
away from crossing arcs, $x$ has enhancements 
$X^-\in \Cs_{R,\color{ForestGreen}p\to1\color{black}}(x)$, $X^+\in \Cs_{R,\color{NavyBlue}p\to0\color{black}}(x)$,  
identical away from $p$, such that 
both $\text{tr}_{\F_2}{X^\pm}$ represent nonzero homology classes. 
If also $G_{x_{\color{gray}A\color{black}}}$ is bipartite, then both $\text{tr}_{\Z}X^\pm$ represent nonzero homology classes.
\end{maintheorem}
\begin{proof}
Construct $X^\pm$ as follows. 
First, label the state circle containing $p$: \color{ForestGreen}1 \color{black} for $X^-$,  \color{NavyBlue}0 \color{black} for $X^+$.  
Next, for both $X^-$ and $X^+$, label all remaining state circles in the zone(s) containing $p$: \color{ForestGreen}1 \color{black} for any state circle in an $\color{gray}A\color{black}$--zone containing $p$,  \color{NavyBlue}0 \color{black} for $B$--. Next, in each zone which abuts the first one (or two, in case $p\in x_{\color{gray}A\color{black}}\cap x_B$), label all remaining circles \color{ForestGreen}1  \color{black} or \color{NavyBlue}0\color{black}, according the type ($\color{gray}A\color{black}$-- or $B$-- respectively) of the zone.  Repeat in this manner, progressing by adjacency, until every state circle of $x$ has been labeled. 

The resulting $X^\pm$ are identical away from 
$p$, with ${{j}_{X^\pm}}=j_x\pm 1$.  Since $X^\pm$ satisfy the hypotheses of Proposition \ref{P:4}, both of $\text{tr}_{\F_2}{X^\pm}$ represent nonzero homology classes.   
No additional effort is required over $R=\Z$, due to the assumption in this case that every $\color{gray}A\color{black}$--zone of $x$ is bipartite.
\end{proof}

%

\section{Remarks and questions}\label{S:final}

A state $x$ need not be essential  in order for $\Cs_{\F_2}(x)$ to be nonzero; indeed, $x$ need not even be adequate.
Consider two examples. 
First, for the trivial diagram of two components, $\bigcirc~\bigcirc$, the homology groups are 
\[\Kh_{\F_2}^{0,-2}=\F_2\cdot\color{ForestGreen}\bigcirc~\bigcirc\color{black},\qquad
\Kh_{\F_2}^{0,0}=\F_2\cdot\color{ForestGreen}\bigcirc~\color{NavyBlue}\bigcirc\color{black}\oplus~\F_2\cdot\color{NavyBlue}\bigcirc~\color{ForestGreen}\bigcirc\color{black} ,\qquad\Kh_{\F_2}^{0,2}=\F_2\cdot\color{NavyBlue}\bigcirc~\color{NavyBlue}\bigcirc\color{black}.\]
Now perform a Reidemeister--2 (R2) move to get the connected diagram $\raisebox{-.025in}
{\includegraphics[height=.15in]{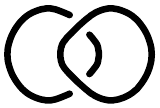}}$.  Each of the four homology generators can still be taken to be the $\color{gray}A\color{black}$-trace of an enhancement of a single state, namely $\raisebox{-.025in}{\includegraphics[height=.15in]{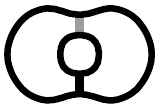}}$ or $\raisebox{-.025in}{\includegraphics[height=.15in]{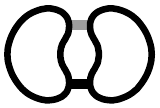}}$. 
Yet, these states are not essential, since their state surfaces are connected and span a split link.  

Second, consider the enhancement $X=\raisebox{-.03in}
{\includegraphics[height=.2in]{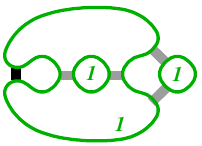}}$ of the state $x=\raisebox{-.03in}{\includegraphics[height=.2in]{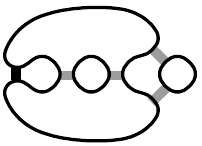}}$ of the diagram $D=\raisebox{-.03in}
{\includegraphics[height=.2in]{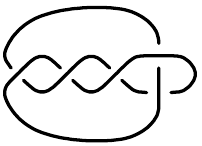}}$ (cf Figure \ref{Fi:NonzeroInessential}); $X$ is a cycle with any coefficients. 
Moreover, $X$ is not exact unless 2 is a unit in $R$, as $\varepsilon\circ \pi_\Bs\circ {d}\equiv 0$ over $\F_2$, where $\Bs$ is an $\F_2$--basis for $\Cs_{\F_2}^{i_X,j_X}(D)$ (cf Figure \ref{Fi:NonzeroInessential}).

\begin{figure}
\begin{center}
$X=\raisebox{-.1in}{\includegraphics[width=.35in]{Fig8X1.pdf}}$
\hspace{.05in}
$D=\raisebox{-.1in}{\includegraphics[width=.35in]{Fig8Link.pdf}}$
\hspace{.05in}
$\Bs=\Bigg\{
\raisebox{-.1in}{\includegraphics[width=.35in]{Fig8X1.pdf},~
\includegraphics[width=.35in]{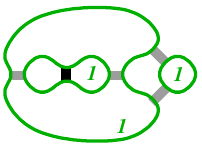},~
\includegraphics[width=.35in]{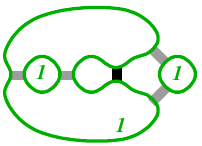},~
\includegraphics[width=.35in]{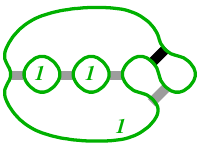},~
\includegraphics[width=.35in]{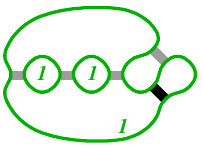}}
\Bigg\}$
\hspace{.05in}
$\Bs'=\Bigg\{
\raisebox{-.1in}{\includegraphics[width=.35in]{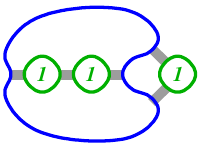},~
\includegraphics[width=.35in]{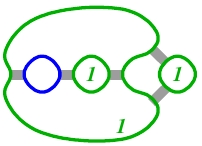},~
\includegraphics[width=.35in]{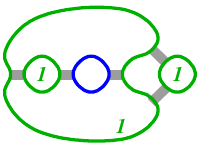},~
\includegraphics[width=.35in]{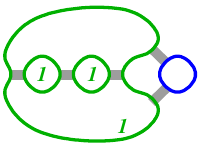}}
\Bigg\}$
\caption{An enhancement $X$  
of an \textit{inessential} 
state of a link diagram $D$; bases $\Bs$ for $\Cs_{\F_2}^{i_X,j_X}(D)$ and $\Bs'$ for $\Cs_{\F_2}^{i_X-1,j_X}(D)$.  
Check that ${d} X=0$ and  $\varepsilon\circ \pi_{\Bs}\circ {d}\equiv 0$ over $\F_2$.  }
\label{Fi:NonzeroInessential}
\end{center}
\end{figure}





Here is an idea for extending the main theorem: establish a class of essential states which are nonzero in two distinct $j$--gradings in Khovanov homology, say over $\F_2$---for simplicity, insist that the initial class must consist only of checkerboard states---and then aim to extend by plumbing. 
 The easiest such class of checkerboard states consists of those which are {\it alternating}; plumbing these gives the adequate homogeneous states. 
To construct a new (non-alternating) essential checkerboard state, consider any (non-alternating) link diagram $D$ which admits no $n\to 0$ wave moves. This means that, whenever $\alpha\subset D$ is a smooth arc whose endpoints are away from crossings and whose interior contains $n$ overpasses and no underpasses, or vice-versa, every arc $\beta\subset S^2$ with the same endpoints as $\alpha$  intersects $K$ in its interior.
Construct either checkerboard surface $F$ for $D$, and replace each of its half-twist crossing bands with a full-twist band in the same sense. (Any band with at least two half-twists in the same sense suffices.)

 The resulting link diagram $D'$ has twice as many crossings as $D$, and the resulting surface $F'$ is an essential, two--sided checkerboard surface with the same euler characteristic as $F$. 
 (To prove essentiality, use Menasco's crossing ball structures, hypothesize a compression disk $\Delta$ for $F'$ which intersects the crossing ball structure minimally, characterize the outermost disks of $\Delta\setminus (S^2\cup C)$, and observe that the only viable configuration for a height one component of $\Delta\cap (S^2\cup C)$ implies that $D$ admitted an $n\to 0$ wave move, contrary to assumption.) 
Does Khovanov homology detect the essential checkerboard states from this construction? 

The simplest non-alternating diagram admitting no $n\to0$ wave moves is a  4--crossing diagram of the trefoil. Following Figure \ref{Fi:DoubledStates}, construct an 8--crossing diagram from this one in the manner just described, and consider $X=\raisebox{-.05in}{\includegraphics[height=.2in]{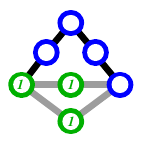}}$, with  $\text{tr}_RX= \raisebox{-.05in}{\includegraphics[height=.2in]{X123.pdf}}-\raisebox{-.05in}{\includegraphics[height=.2in]{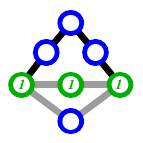}}-\raisebox{-.05in}{\includegraphics[height=.2in]{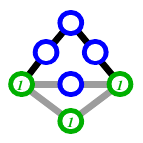}}+\raisebox{-.05in}{\includegraphics[height=.2in]{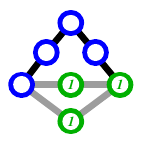}}$, and $Y=\raisebox{-.05in}{\includegraphics[height=.2in]{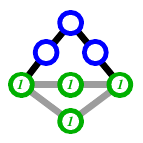}}$.
\begin{figure}
\begin{center}
\scalebox{1.3}{\raisebox{-.15in}{\includegraphics[width=2in]{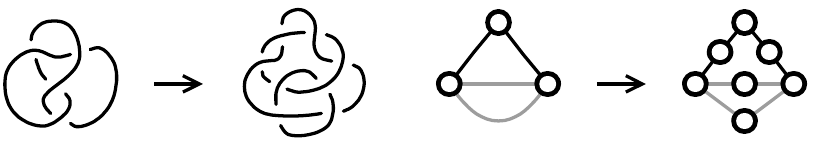}}\hspace{.25in}$\raisebox{-.15in}{\includegraphics[width=.4in]{X123.pdf}}-\raisebox{-.15in}{\includegraphics[width=.4in]{X124.pdf}}-\raisebox{-.15in}{\includegraphics[width=.4in]{X134.pdf}}+\raisebox{-.15in}{\includegraphics[width=.4in]{X234.pdf}}$},\hspace{.25in}\scalebox{1.3}{\raisebox{-.15in}{\includegraphics[width=.4in]{X1234.pdf}}}
\caption{Left: an essential state built by doubling crossings.  Right: Are these cycles exact?}\label{Fi:DoubledStates}
\end{center}
\end{figure}
Both are cycles over $R=\F_2$ and $R=\Z$, but exactness is not so easy. When considering the exactness of a cycle $\text{tr}_RZ$ from an enhancement $Z$ of an adequate homogeneous state, it sufficed to consider $\pi_{[Z]_B}\circ d$. In the case of $X$ from above, 
considering the map $\pi_{\Cs_\Z(x)}\circ d$ proves that $\text{tr}_\Z X$ 
is not a cycle; yet this computation proves nothing regarding $\text{tr}_{\F_2}X$ or $Y$. 
We leave it as an open question whether this construction yields states $x$ with $\Cs_{\F_2}(x)\cap \ker(d)\not\subset \text{image}(d)$, even in this simplest example.  We also ask:

\begin{question}\label{Q:essentialA}
If $x$ is an essential state, does $\Cs_{\F_2}(x)$ always contain a nonzero homology class?
\end{question}
\begin{question}\label{Q:distinguish}
Is there a general method for distinguishing those Khovanov homology classes that correspond to essential states from those that do not? 
\end{question}
%
 %
\begin{question}
Does every link have a diagram with an essential state? A homogeneously adequate state?
\end{question}

\end{document}